%% file: main.tex
\documentclass[Alon2,singlecolor,11pt]{Alon}
\usepackage{fix-cm}
\usepackage[english]{babel}
\usepackage{lmodern}
\usepackage{amssymb}
\usepackage{amsmath}
\usepackage{graphicx}
\usepackage{subfigure}
\usepackage{makeidx}
\usepackage{multicol}

\usepackage{changepage}

\makeindex

\include{preamble}

\usepackage{version}

\title{Group actions on codes in graphs}
\author{Daniel R. Hawtin and Cheryl E. Praeger}

\begin{document}







\setcounter{page}{7} 

\mainmatter

\setcounter{chapter}{4}
\include{chapter}

\include{bibliography}
\printindex
\cleardoublepage
\end{document}

%% file: preamble.tex
\usepackage{amsthm,latexsym,tikz,anysize,enumerate,mathrsfs}
\usepackage{graphicx,float,wrapfig,avant,colortbl,upgreek,array,tikz}
\theoremstyle{plain}
\newtheorem{theorem}{Theorem}[section]
 
\newtheorem{proposition}[theorem]{Proposition} 
\newtheorem{lemma}[theorem]{Lemma}

\theoremstyle{definition} 
\newtheorem{definition}[theorem]{Definition}
\newtheorem{example}[theorem]{Example}
\newtheorem{remark}[theorem]{Remark}

\newtheorem{problem}[theorem]{Problem}

\def\C{{\mathcal C}}

\def\G{{\mathcal G}}
\def\Had{{\mathcal H}}
\def\F{{\mathbb F}}
\def\Z{{\mathbb Z}}
\def\W{{\mathsf W}}
\def\He{{\mathsf H}}
\def\P{{\mathcal P}}
\def\PHad{{\mathcal P\mathcal H}}
\def\B{{\mathcal B}}
\def\PG{{\rm PG}}
\def\AG{{\rm AG}}
\def\Prod{{\rm Prod}}
\def\Rep{{\rm Rep}}
\def\GammaL{\Gamma{\rm L}}
\def\PGammaL{{\rm P}\Gamma{\rm L}}
\def\Gpi{G_{{\mathbf 0},i}^{\Q_i^\times}}

\def\I{{\rm I}}
\def\L{{\mathcal L}}
\def\N{{\mathcal N}}
\def\NR{{\mathcal N\mathcal R}}
\def\M{{\mathcal M}}
\def\O{{\mathcal O}}
\def\Q{{\mathcal Q}}
\def\Qu{{\mathscr Q}}
\def\Quu{{\mathsf Q}}
\def\V{{\mathcal V}}
\def\SS{{\mathcal S}}
\def\S{{\mathcal S}}
\def\T{{\mathcal T}}
\def\R{{\mathcal R}}
\def\SG{{\mathcal S\mathcal G}}
\def\EG{{\mathcal E\mathcal G}}
\def\RM{{\mathcal R\mathcal M}}
\def\PRM{{\mathcal P\mathcal R\mathcal M}}
\def\J{{\mathcal J}}
\def\mg{{\rm M}}
\def\alt{{\rm A}}

\def\GL{{\rm GL}}
\def\GaL{\Gamma {\rm L}}

\def\la{\langle}
\def\ra{\rangle}
\def\SL{{\rm SL}}
\def\s{{\rm S}}
\def\supp{{\rm supp}}
\def\PSL{{\rm PSL}}
\def\ASL{{\rm ASL}}

\def\PGaL{{\rm P}\Gamma {\rm L}}
\def\PSiL{{\rm P}\Sigma {\rm L}}
\def\GaL{\Gamma {\rm L}}

\def\PSU{{\rm PSU}}
\def\Ree{{\rm Ree}}
\def\Sz{{\rm Sz}}

\def\PGU{{\rm PGU}}

\def\HS{{\rm HS}}
\def\Co{{\rm Co}}
\def\mg{{\rm M}}

\def\Sp{{\rm Sp}}

\def\GL{{\rm GL}}
\def\tr{{\rm tr}}
\def\wt{{\rm wt}}

\def\Sym{{\rm Sym}}
\def\soc{{\rm soc}}

\def\Aut{{\rm Aut}}
\def\AGL{{\rm AGL}}
\def\Diag{{\rm Diag}}
\def\rank{{\rm rank}}

\def\la{\langle}
\def\ra{\rangle}

\renewcommand{\b}{\mathbf}

%% file: chapter.tex
\chapter[Group actions on codes in graphs]{Group actions on codes in graphs}
\vspace{-2cm}
 \textsf{ {\Large Daniel R. Hawtin}\quad and\quad {\Large Cheryl E. Praeger} }

\tableofcontents

\section{Introduction: Codes in Graphs}\label{sect:intro}

Traditionally error-correcting codes are modelled as subsets $\C$ of vectors in a finite vector space $V=\mathbb{F}_q^n$ over a field of order $q$, where the vectors are represented as $n$-dimensional row vectors with entries from $\mathbb{F}_q$, and distance between two vectors is the number of entries where they differ. 
The \emph{minimum distance} $\delta$ of $\C$ is the smallest distance between two codewords, and the code is called \emph{perfect} if the balls of radius $\lfloor (\delta-1)/2\rfloor$ partition the space $V$.  Much work was devoted by many researchers to understanding and finding new perfect codes, and eventually, Tiet{\"a}v{\"a}inen~\cite{tietavainen1973nonexistence}, and Zinoviev and Leontiev~\cite{Zinoviev73thenonexistence}, independently showed that  the only non-trivial perfect codes over finite fields are `Hamming-like' codes (perfect single-error correcting codes), and the famous Golay codes  in $\mathbb{F}_2^{23}$ and $\mathbb{F}_3^{11}$, (see for example, the 1975 survey on perfect codes by Van Lint \cite{VL75}). Larger families of codes with desirable properties were sought.  
For example, in 1971, Semakov, Zinoviev, and Zaitsev\cite{SemZinZai71} introduced  a family of codes, properly containing the perfect codes, which they called uniformly packed codes, and which retained strong regularity properties for  `packing' codewords in the vector space.  Then independently, around 1973, Norman Biggs~\cite{Biggs1973289, Biggs1973b} and Philippe Delsarte \cite{delsarte1973algebraic, delsarte1974} suggested a complete change of focus.
Biggs introduced the concept of a perfect code in a graph, while Delsarte developed a general theory of association schemes in coding theory. This led to the notion of {\em codes in graphs}, which is the theme of this chapter. 

The chapter will provide an overview of recent work on codes in graphs which are neighbour-transitive, or have stronger symmetry properties. Most of the graphs considered are distance-regular, and our particular focus is codes in Hamming graphs since the current work on neighbour-transitive codes in Johnson graphs (a particular study advocated by Delsarte) has been recently covered in the 2021 survey by the second author~\cite{praeger2021codes}\footnote{This was the Clay Lecture at the British Combinatorial Conference 2021, and a version of the survey may be downloaded from \texttt{https://www.claymath.org/events/british-combinatorial-conference-2021/}}.  In Section~\ref{sec:ct} we give a brief historical account of completely transitive codes in Hamming graphs and Johnson graphs which motivated the more general developments on neighbour-transitivity, particularly in Hamming graphs, surveyed in Section~\ref{sec:hamming}. We also summarise very recent work on codes in other graphs, such as Kneser graphs and the incidence graphs of generalised quadrangles, as well as state some results for Grassmann graphs and bilinear forms graphs (the $q$-analogues of the Johnson and Hamming graphs).

\subsection{Codes in graphs: neighbour-transitivity and complete transitivity}\label{sec:nt}

A {\em graph} $\Gamma=(V, E)$ consists of a set $V$ of {\em vertices} and a set $E$ of {\em edges} (unordered pairs of vertices), and  
a {\em code} $\C$ in $\Gamma$ is defined as a subset of $V$. {\em Distance} between distinct {\em codewords} (elements of $\C$) is given by the length of a shortest path between them in the graph $\Gamma$, and the {\em minimum distance} $\delta(\C)$ is the minimum distance between distinct codewords.  This viewpoint also suggests a natural measure for the symmetry of a code $\C$ in $\Gamma$. 
Symmetry of the graph $\Gamma$ is measured by its {\em automorphism group} $\Aut(\Gamma)$,  namely the subgroup of permutations of $V$ which leave invariant the edge-set $E$. We take as the   \emph{automorphism group} $\Aut(\C)$ of the code $\C$ the subgroup of $\Aut(\Gamma)$ consisting of all elements that leave $\C$  invariant (setwise).
Delsarte suggested choosing $\Gamma$ to be a {\em distance-regular graph}: that is to say, 
 for any (possibly equal) vertices $v$ and $w$, the number of vertices at distance $j$ from $v$ and at distance $k$ from $w$ depends only upon $j, k$, and the distance between $v$ and $w$.
He defined a special type of code, now 
called a completely-regular code, `which enjoys combinatorial (and often 
algebraic) symmetry akin to that observed for perfect codes' (see \cite[page 1]{Martin04}). Again, disappointingly, not many examples of completely regular codes
were found with good error-correcting properties -- that is to say, with large distance between distinct codewords (see the comments and references in \cite[Section 1.1]{praeger2021codes}). 

A useful notion for exploring a code $\C$ in a graph $\Gamma$ is the \emph{distance partition} of $\C$. This is the partition:
%
%
\begin{equation}\label{def:covrad}
\text{
\begin{minipage}{14cm}
 $\{\C_0,\C_1,\ldots,\C_\rho\}$ of $V$ such that $\C_0=\C$ and, for $i\geq1$, $\C_i$ is the set of all vertices $\gamma$ such that the minimum distance between $\gamma$ and a codeword is equal to $i$. The largest integer $i$ such that $\C_i\ne \emptyset$ is called the  \emph{covering radius}  of $\C$, and is denoted $\rho$ and sometimes $\rho(\C)$.
\end{minipage}
}
\end{equation}
    
\noindent
Since graph automorphisms preserve distance, the automorphism group $\Aut(\C)$ fixes each of the subsets $\C_i$ setwise, and so each $\C_i$ is a union of  $\Aut(\C)$-orbits. A natural, but very strong, symmetry condition to place on a code is to require that each $\C_i$ is a single orbit of $\Aut(\C)$. Codes with this property are called \emph{completely transitive}, and this family of codes was among the first family of codes to be studied in various families of graphs. We discuss some of the results about completely transitive codes in Section~\ref{sec:ct}.

Most recent studies of codes in graphs have focused on strictly larger families than the completely regular codes or the completely transitive codes, where their strict regularity or symmetry conditions have been replaced by more `local' conditions. This new approach perhaps dates back to a discussion between the second author and (Bob) Liebler in 2005 (leading  to~\cite{liebler2014neighbour}). In the context of codes in Johnson graphs, Bob suggested that the stringent regularity conditions imposed for complete regularity, could be replaced by  a `local transitivity' property. This led to the notion of a
\emph{neighbour-transitive code} in an arbitrary graph $\Gamma$, that is to say, a code $\C$ such that 
$\Aut(\C)$ is transitive both on $\C$ and on the set $\C_1$ of the distance partition. The vertices in $\C_1$ are called  {\em code-neighbours} (the non-codewords that are adjacent in $\Gamma$ to some codeword).  More recently, for any positive integer $s\leq \rho(\C)$, a code $\C$ in $\Gamma$ is called  \emph{$s$-neighbour-transitive} if each of $\C_0,\C_1,\dots,\C_s$ is an $\Aut(\C)$-orbit. Thus the neighour-transitive codes are $1$-neighbour-transitive, and the completely transitive codes are $\rho(\C)$-neighbour-transitive. 






 
\section{Fundamental concepts: neighbour-transitive codes}

In this section we present some general concepts and results regarding codes in graphs. We begin with a discussion in arbitrary graphs before specialising to the Hamming graphs and other specific graph families.

\subsection{Parameters and regularity properties for codes in graphs}\label{sec:param}

Let $\Gamma=(V,E)$ be a graph. Note that we will always assume that $\Gamma$ is simple, finite, undirected and connected. Let $\alpha,\beta\in V$. Denote by $d(\alpha,\beta)$ the \emph{distance} in $\Gamma$ between $\alpha$ and $\beta$, that is, the length of the shortest path between $\alpha$ and $\beta$. Define $\Gamma_i(\alpha)=\{\gamma\in V\mid d(\alpha,\gamma)=i\}$. Furthermore, the \emph{ball} $B_i(\alpha)$ of radius $i$ centered at $\alpha$ is defined to be $\bigcup_{j=0}^i \Gamma_j(\alpha)$.

Let $\C$ be a code in a graph $\Gamma=(V,E)$. We refer to elements of $\C$ as \emph{codewords}. A code $\C$ such that $|\C|\leq 1$ or $\C=V(\Gamma)$ is called \emph{trivial}, and we will often be assuming without statement that $\C$ is non-trivial. If $\C'$ is a subset of $\C$ then we say that $\C'$ is a \emph{subcode} of $\C$. The \emph{minimum distance} $\delta$ of $\C$ is the smallest distance between a pair of distinct elements of $\C$. The \emph{error-correction capacity} $e$ of $\C$ is defined to be the largest value of $i$ for which, given distinct $\alpha,\beta\in \C$, the balls $B_i(\alpha)$ and $B_i(\beta)$ are disjoint. These two parameters are related by $e=\lfloor (\delta-1)/2\rfloor$. The \emph{covering radius} $\rho$ of $\C$ is the smallest value of $i$ for which $\bigcup_{\alpha\in \C} B_i(\alpha)=V$. 

\begin{lemma}\label{lemDisjointUnion}
 Let $\C$ be a code with error-correction capacity $e$ in a graph $\Gamma$ and let $i\leq e$. Then the following hold.
 \begin{enumerate}[$(1)$]
     \item For each $\gamma\in \C_i$ there exists a unique $\alpha\in \C$ such that $\gamma\in\Gamma_i(\alpha)$.
     \item $\C_i$ is the disjoint union $\bigcup_{\alpha\in\C} \Gamma_i(\alpha)$.
     \item $\Gamma_i(\alpha)\neq\emptyset$ for all $\alpha\in \C$.
 \end{enumerate}
\end{lemma}

\begin{proof}
All the assertions hold for a trivial code $\C$, with $i=e=0$, so we may assume that $\C$ is non-trivial. Then its minimum distance $\delta=\delta(\C)$ is a positive integer.
 Since every element of $\C_i$ is at distance $i$ from some element of $\C$ we have $\C_i\subseteq \bigcup_{\alpha\in\C} \Gamma_i(\alpha)$. Let $\alpha\in \C$ and $\gamma\in\Gamma_i(\alpha)$. Consider a codeword $\beta\in\C$ such that $\alpha\neq \beta$. If $d(\beta,\gamma)\leq i$, then $\gamma\in B_i(\alpha)\cap B_i(\beta)$ and hence $d(\alpha,\beta)\leq 2i\leq 2e<\delta$, which is  a contradiction. Thus  $d(\beta,\gamma)>i$.  In particular, $\alpha$ is the unique codeword in $\C$ at distance $i$ from $\gamma$, and part (1) holds. This implies moreover that $\gamma\in\C_i$. Thus $\Gamma_i(\alpha)\subseteq  \C_i$ for each $\alpha\in\C$, and hence $\bigcup_{\alpha\in\C} \Gamma_i(\alpha)= \C_i$. Now, if $\beta_1$ and $\beta_2$ are distinct codewords in $\C$, then the assumption $i\leq e$ implies that $B_i(\alpha)\cap B_i(\beta)=\emptyset$ so that $\Gamma_i(\alpha)\cap\Gamma_i(\beta)=\emptyset$. Thus the union is disjoint, and part (2) is proved. 
 
 Consider distinct codewords $\alpha,\beta\in\C$, and let $(\gamma_0,\gamma_1,\dots, \gamma_r)$ be a path in $\Gamma$ of length $r=d(\alpha,\beta)$ from $\alpha=\gamma_0$ to $\gamma_r=\beta$. It follows from the  minimality in the definition of $d(\alpha,\beta)$ that $d(\alpha, \gamma_j)=j$ and $d(\beta, \gamma_{r-j})=j$ whenever $0\leq j\leq r$. If $r\leq i$, then taking $j=r$ we have $\beta\in B_i(\alpha)\cap B_i(\beta)$, which is a contradiction, since $i\leq e$. Hence $r>i$ and  $\gamma_i\in \Gamma_i(\alpha)$, proving part (3). 
 \end{proof}

\begin{definition}\label{CRdefinition}
 For a code $\C$ in a graph $\Gamma$ and a non-negative integer $s$ at most the covering radius $\rho$, $\C$ is said to be \emph{$s$-regular} if for each $i\in\{0,1,\ldots,s\}$ there exist non-negative integers $a_i$, $b_i$ (if $s<\rho$), and $c_i$ (if $i>0$), such that for each vertex $\alpha\in \C_i$ there are precisely:
 \begin{enumerate}[(1)]
     \item $a_i$ vertices in $\Gamma_1(\alpha)\cap \C_i$,
     \item $b_i$ vertices in $\Gamma_1(\alpha)\cap \C_{i+1}$, and,
     \item $c_i$ vertices in $\Gamma_1(\alpha)\cap \C_{i-1}$,
 \end{enumerate}
 and $a_i,b_i,c_i$ depend only on $i$, and not on the particular choice of $\alpha$. If $s=\rho$ then $\C$ is said to be   \emph{completely regular}.
\end{definition}

It is worth mentioning that a graph $\Gamma$ is distance-regular, according to the usual definition, if and only if, for every vertex $\alpha$ of $\Gamma$, the singleton set $\{\alpha\}$ is a completely regular code in $\Gamma$.

\subsection{Symmetry of codes in graphs}\label{sec:sym}

The \emph{symmetric group} $\Sym(V)$ is the group of all permutations of $V$, and, for $\alpha\in V$ and $g\in G$, we write $\alpha^g$ for the image of $\alpha$ under $g$. 
Each element $g\in\Sym(V)$ permutes subsets of $V$ in a natural way, namely for $U\subseteq V$, the image $U^g$ of $U$ under $g$ is the set $\{\alpha^g\mid\alpha\in U\}$ of images for the elements of $U$.
The \emph{setwise stabiliser} of $U$ is the set $\Sym(V)_U:=\{ g\in\Sym(V)\mid U^g=U\}$, and this is a subgroup of $\Sym(V)$. In particular, for each $i< |V|$, $g$ permutes the $i$-element subsets of $V$ among themselves and, for a set $E$ of $i$-element subsets, we say that $g$ \emph{leaves $E$ invariant} if $U^g\in E$ for all $U\in E$. 
We often deal with \emph{permutation groups on $V$}, that is, subgroups $G\leq\Sym(V)$. A permutation group is \emph{transitive} on $V$ if, for all $u,v\in V$, there exists $g\in G$ such that $u^g=v$. For $1<i<v$, we say that $G$ is \emph{$i$-transitive} on $V$ if $G$ is transitive on $V$ and, for $v\in V$, the stabiliser $G_v$ is $(i-1)$-transitive on $V\setminus\{v\}$. Also $G$ is said to be \emph{$i$-homogeneous} on $V$ if $G$ is transitive on the set of $i$-element subsets of $V$. 

Let $\Gamma=(V,E)$ be a graph. Then an \emph{automorphism} of $\Gamma$ is a permutation $g\in\Sym(V)$ such that $g$ leaves the edge set $E$ invariant. The set of all automorphisms of $\Gamma$ forms a subgroup $\Aut(\Gamma)$ called the \emph{automorphism group} of $\Gamma$.
\begin{equation}\label{def:autC}
\text{
\begin{minipage}{14cm}
If $\C$ is a code in $\Gamma$ then, as introduced in Section~\ref{sect:intro}, the \emph{automorphism group} of $\C$ is the setwise stabiliser  $\Aut(\C)$ of $\C$ in $\Aut(\Gamma)$, that is to say, $\Aut(\C)=\Sym(V)_\C\cap \Aut(\Gamma)$.
\end{minipage}
}
\end{equation}

Two codes are \emph{equivalent} if there exists an automorphism of $\Gamma$ mapping one to the other. Note that equivalent codes have many of the same properties, for instance, the same minimum distance, the same covering radius and isomorphic automorphism groups. Hence, we will often be interested in codes only up to equivalence.

The following concepts are the main focus of this chapter and may be viewed as algebraic analogues of Definition~\ref{CRdefinition}.

\begin{definition}\label{defsneighbourtrans}
 Let $\C$ be a code with covering radius $\rho$ in a graph $\Gamma$, let $G\leq \Aut(\C)$, and let $s\in\{1,\ldots,\rho\}$. Then we make the following definitions.
 \begin{enumerate}[(1)]
  \item $\C$ is \emph{$(G,s)$-neighbour-transitive} if $G$ acts transitively on each of the sets $\C,\C_1,\ldots, \C_s$.
  \item $\C$ is \emph{$G$-neighbour-transitive} if $\C$ is $(G,1)$-neighbour-transitive.
  \item $\C$ is \emph{$G$-completely transitive} if $\C$ is $(G,\rho)$-neighbour-transitive.
 \end{enumerate}
 Moreover, we say that $\C$ is \emph{neighbour-transitive}, \emph{$s$-neighbour-transitive}, or \emph{completely transitive}, respectively, if $\C$ is $\Aut(\C)$-neighbour-transitive, $(\Aut(\C),s)$-neighbour-transitive, or $\Aut(\C)$-completely transitive, respectively.
\end{definition}

It turns out that many famous codes have these symmetry properties. We mention several of them in Example~\ref{exam:introHammingExamples}. We also give  in Example~\ref{exam:c2} a simple example of an explicit infinite family of completely transitive codes.


\begin{example}\label{exam:introHammingExamples}
 The following well-known codes are completely-transitive codes in the Hamming graphs (see Section~\ref{sec:HammingPrelim} for introductory material relating to the Hamming graphs and see \cite[Sections~5.1 and 5.2]{borges2019completely} for more examples):
 \begin{enumerate}[(1)]
  \item The perfect binary Golay code is a $12$-dimensional vector-subspace of $\F_2^{23}$ with minimum distance $7$ and covering radius $3$, and its extended code is a $12$-dimensional vector-subspace of $\F_2^{24}$ with minimum distance $8$ and covering radius $4$. These codes are $G$-completely transitive for $G=T\rtimes \mg_{23}$ and $G=T\rtimes \mg_{24}$, respectively, where $T$ is the group of translations by codewords in each case; see \cite[p.~199]{sole1990completely}. 
  This extended Golay code was used by NASA's Voyager spacecraft to send back to earth hundreds of  colour pictures of Jupiter and Saturn in their 1979, 1980, and 1981 fly-bys. Error correction was vital to data transmission since memory constraints dictated offloading data virtually instantly leaving no second chances, and the data needed to be transmitted within a constrained telecommunications bandwidth.\footnote{See \texttt{en.wikipedia.org/wiki/Binary\_Golay\_code}.}

  \item The perfect ternary Golay code is a $6$-dimensional vector-subspace of $\F_3^{11}$ with minimum distance $5$ and covering radius $2$ and its extended code is a $6$-dimensional vector-subspace of $\F_3^{12}$ with minimum distance $6$ and covering radius $3$. These codes are $G$-completely transitive for $G=T\rtimes (2.\mg_{11})$ and $G=T\rtimes (2.\mg_{12})$, respectively, where $T$ is the group of translations by codewords in each case; see \cite[p.~653]{Giudici1999647}.
  \item The Nordstrom--Robinson code is a non-linear code consisting of $256$ codewords with minimum distance $5$ and covering radius $3$ in $\F_2^{15}$ and its extended code is a non-linear code consisting of $256$ codewords
  with minimum distance $6$ and covering radius $4$ in $\F_2^{16}$. These codes are $G$-completely transitive for $G\cong 2^5\rtimes \alt_8\cong 2^5\rtimes \GL_4(2)$ and $G\cong 2^5\rtimes \AGL_4(2)$, respectively; see \cite{gillespie2012nord}.
 \end{enumerate}
\end{example}

\begin{example}\label{exam:c2}
 Let $\Gamma$ be the graph with vertex set $\Z_{2n}$ and vertices $i,j$ defined to be adjacent if $i-j=\pm 1$ ({\em i.e.}, $\Gamma$ is a cycle of length $2n$) and let $\C=\{0,n\}$. Then $\Aut(\C)$ is generated by the rotation $i\mapsto i+n$ (for $i\in\Z_{2n}$), and the reflection $i\mapsto -i$ (for $i\in\Z_{2n}$). Hence $\Aut(\C)\cong C_2^2$ and $\Aut(\C)$ acts transitively on $\C$. We also have  $\C_i=\{\pm i,n\pm i\}$ for each $i$ satisfying $1\leq i\leq n/2$. Note that when $1\leq i< n/2$ the cardinality $|\C_i|=4$, but when $n$ is even we have $|\C_{n/2}|=2$. In all cases $\Aut(\C)$ acts transitively on $\C_i$ for each $i$ satisfying $1\leq i\leq n/2$, and thus $\C$ is completely transitive.
\end{example}

Lemma~\ref{lem:largeDeltaDistTrans}, below, gives two additional equivalent conditions for $(G,s)$-neighbour-transitivity for integers $s$ at most the error-correction capacity. It is most useful for $s$-neighbour-transitive codes with `large' minimum distance $\delta$, namely  $\delta\geq 2s+1$. Applications of this result are two-fold. Firstly, it is often simpler to prove $s$-neighbour-transitivity of a code in terms of the `local action' of the stabiliser of a codeword $\alpha$ on the ball $B_s(\alpha)$ (Lemma~\ref{lem:largeDeltaDistTrans}(2)). Secondly, Lemma~\ref{lem:largeDeltaDistTrans} often allows, in the case of a specific graph, for us to prove structural results. We will see examples of such structural results in later sections.

\begin{lemma}\label{lem:largeDeltaDistTrans}
 Let $\C$ be a code  in a graph $\Gamma$ such that $\C$ has error-correction capacity $e\geq 1$, let $G\leq\Aut(\C)$, let $\alpha\in \C$, and let $s$ be an integer such that $1\leq s\leq e$. Then the following are equivalent.
 \begin{enumerate}[$(1)$]
     \item $\C$ is $(G,s)$-neighbour-transitive.
     \item $G$ acts transitively on $\C$ and, for each $i\in\{1,\ldots,s\}$, the stabiliser $G_\alpha$ is transitive on $\Gamma_i(\alpha)$.
     \item For each $i\in\{1,\ldots,s\}$, $G$ acts transitively on the set $\{(\beta,\gamma)\mid \beta\in\C,\gamma\in \Gamma_i(\beta)\}$. 
 \end{enumerate}
\end{lemma}

\begin{proof}
 Suppose that part (1) holds. Then, by Definition~\ref{defsneighbourtrans}, $G$ acts transitively on $\C$ and transitively on $\C_i$, for each $i\in\{1,\ldots,s\}$. Let $\gamma_1,\gamma_2\in\Gamma_i(\alpha)$. Then, by Lemma~\ref{lemDisjointUnion}(2), $\gamma_1,\gamma_2\in \C_i$ and so there exists $g\in G$ such that $\gamma_1^g=\gamma_2$. By Lemma~\ref{lemDisjointUnion}(1), $\alpha$ is the unique element of $\C$ such that $d(\alpha,\gamma_1)=d(\alpha,\gamma_2)=i$, and hence $\alpha^g=\alpha$, that is, $g\in G_\alpha$. Thus $G_\alpha$ acts transitively on $\Gamma_i(\alpha)$ and so part (1) implies part (2).
 
 Suppose that part (2) holds. Let $1\leq i\leq s$ and let $\beta\in \C$ and $\gamma\in\Gamma_i(\beta)$. Since $G$ acts transitively on $\C$, there exists $g_1\in G$ such that $(\beta,\gamma)^{g_1}=(\alpha,\gamma')$ for some $\gamma'\in\Gamma_i(\alpha)$. Furthermore, $G_\alpha$ acts transitively on $\Gamma_i(\alpha)$, and hence there exists $g_2\in G_\alpha$ such that $(\gamma')^{g_2}=\gamma_1$. Thus, $(\beta,\gamma)^{g_1g_2}=(\alpha,\gamma_1)$ and so part (2) implies part (3).
 
 Finally, suppose that part (3) holds. Let $i\in\{1,\ldots,s\}$, let $\nu_1,\nu_2\in\C_i$ and, as in Lemma~\ref{lemDisjointUnion}, let $\beta_1,\beta_2$ be the unique elements of $\C$ such that $d(\beta_1,\nu_1)=d(\beta_2,\nu_2)=i$. Then $(\beta_j,\nu_j)$ (for $j=1,2$) lies in the set of pairs in part (3), and so there exists an element $h\in G$ such that $(\beta_1,\nu_1)^h=(\beta_2,\nu_2)$. Thus $\nu_1^h=\nu_2$ and $G$ acts transitively on $\C_i$. Also by Lemma~\ref{lemDisjointUnion} part (3), for $\beta_1,\beta_2\in\C$, there exist $\nu_1,\nu_2$ such that,  for each $j=1,2$, $\nu_j\in\Gamma_i(\beta_j)$ so  $(\beta_j,\nu_j)$ lies in the set of pairs in part (3). Thus by part (3), we have  $(\beta_1,\nu_1)^h=(\beta_2,\nu_2)$ for some $h\in G$, and it follows that $G$ is transitive also on $\C$. Thus part (3) implies part (1).
\end{proof}

The case $s=1$ of Lemma~\ref{lem:largeDeltaDistTrans} follows from \cite[Theorem 1.2]{liebler2014neighbour}, noting that $\delta(\C)\geq3$ is equivalent to error capacity $e\geq1$. 
A $G$-neighbour-transitive code $\C$ with the property of Lemma~\ref{lem:largeDeltaDistTrans}(2) is called \emph{strongly-incidence-transit\-ive}, and the theory of strongly incidence transitive codes in Johnson graphs is developed in \cite{liebler2014neighbour}. The expository chapter \cite{praeger2021codes} gives a recent account focusing especially on links between such codes and a family of combinatorial designs called Delandtsheer designs.

In the following proposition we assume the conclusion Lemma~\ref{lem:largeDeltaDistTrans}(2) holds with $s=2$. 

\begin{proposition}\label{prop:coveringRadiusLessThan2}
 Let $\C$ be a non-trivial code with covering radius $\rho$ and minimum distance $\delta$ in a connected graph $\Gamma$. Suppose that $G\leq \Aut(\C)$ such that $G$ acts transitively on $\C$ and, for $\alpha\in\C$,  $G_{\alpha}$ acts transitively on each of $\Gamma_1(\alpha)$ and $\Gamma_2(\alpha)$. Then one of the following holds:
 \begin{enumerate}[$(1)$]
  \item $\rho\geq 2$;
  \item $\rho=1$, $\delta=3$ and $\C$ is a perfect code.
 \item $\rho=1$, $\delta=2$, and either
 \begin{enumerate}[$(a)$]
     \item $\Gamma$ is bipartite and $\C$ is one of the biparts; or
     \item for every pair $\mu,\nu\in\C_1$ with $d(\mu,\nu)=1$ we have that $\Gamma_1(\mu)\cap\C=\Gamma_1(\nu)\cap\C$.
 \end{enumerate}
  \end{enumerate}
\end{proposition}

\begin{proof}
 If $\rho=0$ then $\C=V(\Gamma)$ is a trivial code, but since this is not the case we have $\rho\geq 1$. If $\rho\geq 2$ then part (1) holds. Hence, we may assume that $\rho=1$. Since $\Gamma$ is connected, $G$ acts transitively on $\C$ and $G_\alpha$ acts transitively on $\Gamma_1(\alpha)$, it follows that $\C_1=\bigcup_{\beta\in\C}\Gamma_1(\beta)$. In particular $\delta\geq 2$. If $\delta\geq 3$, then this union is disjoint, and since $\rho=1, |\C|\geq 2$ and $\Gamma$ is connected, there must be an edge between some vertex of $\Gamma_1(\beta)$ and some vertex of $\Gamma_1(\beta')$ for some distinct codewords $\beta$ and $\beta'$, and hence $d(\beta, \beta')=3$, so $\delta =3$.   This implies that $\C$ has error-correction capacity $e=1$, and any pair of balls of radius $1$ centered at distinct codewords is disjoint. Moreover, since $\rho=1$, the vertex set $V(\Gamma)=\C\cup \C_1$, and hence the set of balls of radius $1$ centered at the codewords of $\C$ partitions $V(\Gamma)$. Thus $\C$ is perfect, as in part (2).
 
Thus we may assume that $\delta=2$. Then, since $G$ acts transitively on $\C$ and $G_\alpha$ acts transitively on $\Gamma_2(\alpha)$, it follows, for each codeword  $\beta\in \C$, that $\Gamma_2(\beta)$ is contained in $\C$. If there are no edges between distinct vertices of $\C_1$, then 
 all edges of $\Gamma$ are incident with a vertex of $\C$ and a vertex of $\C_1$, that is to say, $\Gamma$ is bipartite and $\C, \C_1$ form a bipartition as in part (3)(a).  
Hence we may assume that there exists some pair $\mu,\nu\in\C_1$ such that $d(\mu,\nu)=1$. For any such pair $\mu, \nu$, suppose that $\gamma\in\Gamma_1(\mu)\cap\C$. Then $(\gamma, \mu,\nu)$ is a path of length $2$ in $\Gamma$ so $d(\gamma,\nu)\leq 2$. We have shown that $\Gamma_2(\gamma)\subseteq\C$, and since $\nu\in\C_1$ this implies that $d(\gamma,\nu)=1$, that is, $\gamma\in\Gamma_1(\nu)\cap\C$ also. A similar argument holds with $\mu$ and $\nu$ interchanged, and hence $\Gamma_1(\mu)\cap\C=\Gamma_1(\nu)\cap\C$ and part (3)(b) holds.
\end{proof}

In Section~\ref{sec:poly} we apply Proposition~\ref{prop:coveringRadiusLessThan2} to codes in Hamming graphs (see Remark~\ref{rem:smallCovRadHamming}). In that case Proposition~\ref{prop:coveringRadiusLessThan2}(3)(b) never holds, so we are able to make much stronger conclusions. We pose the following problem related to this. 

\begin{problem}\label{prob:covRadOne}
 Investigate codes satisfying Proposition~\ref{prop:coveringRadiusLessThan2}(3)(b). 
\end{problem}

\subsection{Elusive codes}\label{sec:elus}

In this subsection we make a short commentary on the concept of neighbour-transitivity. Let $\C$ be a code in a connected graph $\Gamma$ such that the minimum distance $\delta(\C)\geq 3$ so, by Lemma~\ref{lemDisjointUnion},  the set of code neighbours is the disjoint union $\C_1=\cup_{\alpha\in\C}\Gamma_1(\alpha)$.  It turns out that $\C_1$ determines the code $\C$ if $\delta(\C)$ is large enough and if the graph $\Gamma$ is \emph{reduced}, that is, if: 
\begin{equation}\label{def:red}
    \text{$\Gamma_1(\alpha)=\Gamma_1(\alpha')$ if and only if $\alpha=\alpha'$.}
\end{equation}
We note that all the graphs we consider in the chapter are reduced in this sense. 


\begin{lemma}\label{lem:c1c}
    Let $\C_1$ be the set of code-neighbours of a non-trivial code in a connected regular reduced  graph $\Gamma$. If $\delta(\C)\geq 5$ then $\C=\{ \alpha\in V(\Gamma)\mid \Gamma_1(\alpha)\subseteq \C_1 \}$, and hence $\C_1$ determines $\C$.
\end{lemma}

\begin{proof}
    Let $X:=\{ \alpha\in V(\Gamma)\mid \Gamma_1(\alpha)\subseteq \C_1\}$. Since $\delta(\C)\geq 5$, $\C_1$ contains $\Gamma_1(\alpha)$ for each codeword $\alpha$ and hence $\C\subseteq X$. We claim that equality holds. Suppose to the contrary that $\alpha\in X\setminus\C$, and let $\beta, \beta'$ be distinct vertices in $\Gamma_1(\alpha)\subseteq \C_1$. By the definition of $\C_1$ there are codewords $\gamma, \gamma'\in \C$ such that $(\gamma, \beta, \alpha, \beta',\gamma')$ is a path in $\Gamma$ of length $4$. Since the minimum distance $\delta(\C)\geq 5$, it follows that $\gamma=\gamma'$. For fixed $\alpha, \beta$ and $\gamma$, letting $\beta'$ range over $\Gamma_1(\alpha)$, we see that $\Gamma_1(\alpha)=\Gamma_1(\gamma)$, which is a contradiction since $\Gamma$ is reduced.
    Thus $X=\C$.
\end{proof}

It follows from Lemma~\ref{lem:c1c} that, under the conditions of that lemma, the setwise stabilisers in $\Aut(\Gamma)$ of $\C$ and of $\C_1$ are equal, that is to say, $\Aut(\C)=\Aut(\C_1)$. Then since $\C_1$ is the disjoint union $\cup_{\alpha\in\C}\Gamma_1(\alpha)$, if $\Aut(\C_1)$ is transitive on $\C_1$ then it must also be transitive on $\C$. Thus $\C$ is neighbour-transitive if and only if $\Aut(\C_1)$ is transitive on $\C_1$; and this would be a simplification of Definition~\ref{defsneighbourtrans} of neighbour-transitivity.

However for smaller $\delta(\C)$, it is possible for $\C_1$ to be the set of code neighbours of more than one code. Such a code  $\C$ is said to be \emph{elusive}. An infinite family of elusive neighbour-transitive codes in binary Hamming graphs (see Definition~\ref{defHamming}) was described by Gillespie and the second author in \cite[Section 5]{ntrcodes}, and the smallest code in the family is given in Example~\ref{ex:elus}.

\begin{example}\label{ex:elus}
    Let $\Gamma= H(4,2)$ (see Definition~\ref{defHamming}) and write the vertex set as $V\Gamma = F\times F$ with $F=\mathbb{F}_2^2$. Define
    \begin{align*}
        \C &:= \{(0,0,0,0), (1,1,1,1) \}\\
        \mathcal{C'} &:= \{(0,0,0,0), (1,0,1,0), (0,1,0,1), (1,1,1,1)  \}\\
        \mathcal{X} &:= \{(\beta,\beta'), (\beta',\beta)\mid \beta\in\{(0,0), (1,1)\}, \beta'\in\{(0,1), (1,0)\} \}.
    \end{align*}
    Then $\C$ and $\mathcal{C'}$ are both linear codes in $\Gamma$, with $\delta(\C)=4$ and $\delta(\mathcal{C'})=2$. Further, $\mathcal{X}$ is the set of code neighbours of both of the codes $\C$ and $\mathcal{C'}$; and the code $\C$ is neighbour-transitive.
\end{example}

Two new constructions for infinite families of elusive neighbour-transitive codes in $H(n,q)$ were given by Gillespie and the authors in \cite[Sections 3.1 and 3.2]{elusive}, this time with minimum distance $3$. The constructions produce \emph{elusive pairs} $(\C, X)$, where $\C$ is an elusive code and $X\leq \Aut(\C_1)$ such that $X$ does not fix $\C$ setwise. For these
examples there were precisely two distinct images of $\C$ under elements of $X$, but very recent constructions by the first author in \cite[Theorem 1.1 and 1.2]{selusive} show that the number of images can be arbitrarily large, answering \cite[Question 1.4]{elusive}. 

A \emph{spherical bitrade} in a graph $\Gamma$ is a pair $(\C,\C')$ of codes in $\Gamma$ with the property that for any $\alpha\in V(\Gamma)$ we have  $|\Gamma_1(\alpha)\cap\C|=|\Gamma_1(\alpha)\cap\C'|\in\{0,1\}$. Spherical bitrades are investigated in  \cite{mogilnykh2020bitrades} in relation to a type of switching construction for obtaining new perfect codes from known ones. If $\C$ is an elusive code with $\delta(\C)\geq3$, and $g\in\Aut(\C_1)\setminus\Aut(\C)$, then $(\C,\C^g)$ is a spherical bitrade. As is pointed out in \cite{mogilnykh2020bitrades}, the examples of spherical bitrades in $H(q,q)$ given in \cite[Theorem~2]{mogilnykh2020bitrades} were first constructed as elusive codes in \cite[Example~1]{elusive}. Moreover, one application of the product construction of \cite[Theorem~1]{mogilnykh2020bitrades} is to produce spherical bitrades in $H(kq,q)$ for arbitrary $k$; examples of elusive codes in $H(kq,q)$ were first given in \cite[Lemma~3.9]{elusive}. It is worth noting also that if elusive codes are used as input for \cite[Theorem~1]{mogilnykh2020bitrades} then the resulting codes are also elusive.

An extension of the concept of an elusive code relevant to $s$-neighbour-transitive codes was introduced in \cite{selusive}: a code $\C$ in a graph $\Gamma$ is said to be \emph{$s$-elusive} if there exists a  code $\C'$ distinct from $\C$, but equal to the image of $\C$ under an element of $\Aut(\Gamma)$, such that the sets of $s$-neighbours of $\C$ and $\C'$ are the same. Interesting new examples were found in \cite[Theorem 1.1]{selusive} for $s=1,2,3$ developed from Reed--Muller codes, Preparata codes, and binary Golay codes. It would be interesting to know of examples of $s$-elusive codes in graphs other than the Hamming graphs.

\begin{problem}\label{prob:sElusive}
 Find $s$-elusive, neighbour-transitive codes ($s\geq1$) in other interesting distance-regular graphs, or show that none exist.  
\end{problem}


\subsection{$s$-Neighbour-transitive codes and $s$-distance-transitive graphs}\label{sec:sdist}

There is a beautiful link between $(G,s)$-neighbour-transitive codes and \emph{$(G,s)$-distance-transitive graphs}, namely connected graphs $\Gamma=(V,E)$ for which $G\leq \Aut(\Gamma)$ is transitive on the distance sets $\Gamma_i:=\{(\alpha,\beta)\mid d(\alpha,\beta)=i\}$, for $i=0, 1,\dots, s$. The link involves a \emph{quotient graph} $\Gamma_N$ modulo a normal subgroup $N\unlhd G$: the quotient $\Gamma_N$ is the graph with vertices the $N$-orbits in $V$ such that distinct $N$-orbits $U, U'$ form an edge of $\Gamma_N$ provided there exist $\alpha\in U$ and $\beta\in U'$ such that $\{\alpha,\beta\}\in E$. For a vertex $\alpha\in V$, we denote the $N$-orbit containing $\alpha$ by $\alpha^N :=\{\alpha^x\mid x\in N\}$.

\begin{proposition}\label{p:sdt}
 Let $\Gamma=(V,E)$ be a graph and $G\leq\Aut(\Gamma)$ such that $G$ acts transitively on $V$, and let $N\lhd G$ be intransitive on $V$. Suppose further that, for some $\alpha\in V$, the set $\C=\alpha^N$ is a $(G_{\C},s)$-neighbour-transitive code in $\Gamma$. Then the quotient graph $\Gamma_N$ is $(G/N,s)$-distance-transitive.
\end{proposition}

\begin{proof}
Since $N\lhd G$ and $G$ is transitive on $V$, it follows that $G$ acts transitively on the set of $N$-orbits in $V$, that is to say, the vertex set $V(\Gamma_N)$ of $\Gamma_N$, see \cite[Lemma 2.20]{PS2018}. As $N$ leaves each of its orbits invariant, $N$ is contained in the kernel of this $G$-action and the quotient group $G/N$ acts vertex-transitively on $\Gamma_N$. Recall that $\C=\alpha^N$ is one of the $N$-orbits, and as a code in $\Gamma$, it is $(G_{\C},s)$-neighbour-transitive, by assumption. 

 
 Let $U, U'$ be adjacent vertices of $\Gamma_N$, and let $\{\alpha_1,\alpha_2\}\in E$ such that $\alpha_1\in U$ and $\alpha_2\in U'$. Since $N$ acts transitively on $U$ and on $U'$, it follows that every vertex of $U$ is adjacent in $\Gamma$ to some vertex in $U'$, and every vertex of $U'$ is adjacent in $\Gamma$ to some vertex in $U$. 
 
 Let $X, Y$ be vertices of $\Gamma_N$ at distance $i$ in $\Gamma_N$ from $\C$, where $1\leq i\leq s$. Then, arguing as in the preceding paragraph, there exist vertices $\beta\in X$ and $\gamma\in Y$ such that there are paths of length $i$ in $\Gamma$ from $\alpha$ to each of $\beta$ and $\gamma$. Suppose that there exists a codeword $\alpha'\in\C$ such that there is a path of length $j$  from $\alpha'$ to $\beta$ with $0<j<i$. Then  $\C$ and $X$ would be at distance less than $i$ in $\Gamma_N$, which is not the case. Hence each codeword of $\C$ has distance at least $i$ from $\beta$, and it follows that $\beta\in \C_i$. By an identical argument $\gamma\in\C_i$.  Since $\C$ is $(G_{\C},s)$-neighbour-transitive,  there exists an element $g\in G_{\C}$ such that $\beta^g=\gamma$. Since $G$ permutes the $N$-orbits among themselves, this means that $g$ maps the $N$-orbit $X$ containing $\beta$ to the $N$-orbit $Y$ containing $\gamma$. Thus $G_{\C}$ induces a transitive action on the vertices of $\Gamma_N$ at distance $i$ from $\C$. Since we have already shown that $\Gamma_N$ is vertex-transitive, and since $N$ is in the kernel of the $G$-action on $\Gamma_N$, it follows that $G/N$ acts transitively on the set of ordered pairs of $\Gamma_N$-vertices at distance $i$. Finally since this holds for all $i\leq s$,  the result follows.  
\end{proof}

On the one hand this observation can be used to produce $s$-distance-transitive graphs from certain $s$-neighbour-transitive codes in Hamming graphs:

\begin{example}\label{ex:linearquotients}
 Let $\C$ be an $s$-neighbour-transitive and linear code in the Hamming graph $H(n,q)$ (see Section~\ref{sec:HammingPrelim} for the definitions of `linear' and $H(n,q)$). Moreover, let $T_\C$ be the group of translations by codewords of $\C$ and let $H=\Aut(\C)_{\b 0}$ be the stabiliser of the codeword ${\b 0}\in\C$. Then we may apply Proposition~\ref{p:sdt} with $N=T_\C$, so that $\C={\b 0}^N$, and $G=T_{V(\Gamma)}\rtimes H$, where $T_{V(\Gamma)}\cong\F_q^n$ is the group of translations by all vectors in $V(\Gamma)$, to obtain a $(G/N,s)$-distance-transitive graph $\Gamma_N$. For examples of such codes $\C$, with $s=2$, see Theorem~\ref{binaryx2ntchar}(3) and Propositions~\ref{prop:moduleCodesAffineTopBinary}--\ref{prop:moduleCodesUnitary}.
\end{example}

On the other hand, there are many $(G,s)$-distance-transitive graphs $\Gamma$ known for which the group $G$ has a nontrivial vertex-intransitive normal subgroup $N$, and the question arising here is whether an $N$-orbit provides an  example of an $s$-neighbour-transitive code: 

\begin{example}\label{ex:cover}
 \begin{enumerate}[(1)]
  \item Let $\Gamma$ be a $(G,2)$-arc-transitive graph\footnote{That is, $G$ acts transitively on the set of all paths $(u,v,w)$ in $\Gamma$ with $u\neq w$.} and let $N$ be a normal subgroup of $G$ with at least $3$ orbits on $V(\Gamma)$. It was shown by the second author in \cite[Theorem 4.1]{praeger1993qp} that the quotient $\Gamma_N$, as defined above, is $(G/N,2)$-arc-transitive and the graph $\Gamma$ is a  cover of  $\Gamma_N$. This means that, for each $\alpha\in V(\Gamma)$, the code $\C=\alpha^N$ in $\Gamma$ has minimum distance at least the girth of $\Gamma_N$. Note that
  \begin{center}
      $\rho(\C)=1$ \quad $\Longleftrightarrow$\quad
     $\Gamma_N$ is a complete graph
     \quad $\Longleftrightarrow$\quad 
     $\delta(\C)=3$.
  \end{center}
  In all other cases 
  $\Gamma_N$ has girth at least $4$, so $\delta(\C)\geq4$, $\rho(\C)\geq2$, and $\C$ is $(G_\C, 2)$-neighbour-transitive. 
  \item Let $\Gamma$ be an antipodal $G$-distance-transitive graph with diameter $\delta$. Then an antipodal block $\C$ (that is, a maximal set of vertices mutually at distance $\delta$) is a code in $\Gamma$ with minimum distance $\delta$, and usually $\C$ is an orbit of a normal subgroup of $G$. For example, $(G_\C, 1)$-neighbour-transitive codes $\C$ with $\delta=3$ may be obtained in this way from antipodal distance-transitive covers of complete graphs, and these graphs have been classified in \cite{GLP}. Similarly $(G_\C, 2)$-neighbour-transitive codes $\C$  with $\delta=4$ arise in this way from antipodal distance-transitive covers of complete bipartite graphs, which are classified in \cite{ILPP}.
 \end{enumerate}
\end{example}

Regarding the examples in Example~\ref{ex:cover}(1), we observe that for a $(G,2)$-arc-transitive graph $\Gamma$, a vertex-stabiliser $G_\alpha$ is $2$-transitive on $\Gamma_1(\alpha)$ \cite[Lemma 9.4]{praeger1997NATO} and in particular is primitive, so $\Gamma$ is $G$-locally-primitive. The 2-arc-transitivity condition can be weakened to local primitivity, and each code of the form $\C=\alpha^N$, for an intransitive normal subgroup $N$ of $G$,  will be neighbour-transitive with $\delta\geq 3$ (since $\Gamma$ covers $\Gamma_N$ by \cite[Theorem 10.2]{praeger1997NATO}). Weakening the condition further to $\Gamma$ being $G$-locally-quasiprimitive will yield neighbour-transitive codes $\C=\alpha^N$ with $\delta\geq 2$, see \cite[Theorem 1.3]{LPVZ2002locqp}.

In seeking examples, if we start with a $(G,s)$-distance-transitive graph $\Gamma$ and a normal subgroup $N\lhd G$ with at least three vertex-orbits, then the normal quotient $\Gamma_N$ is always $(G/N,s)$-distance-transitive \cite[Theorem 1.1]{DGLP2012locsdt}. It would be interesting to have examples for Proposition~\ref{p:sdt} where the graph $\Gamma$ is not itself $(G,s)$-distance-transitive. 

\begin{problem}\label{prob:sdt}
 Find graph-group pairs $(\Gamma, G)$ such that $G\leq \Aut(\Gamma)$ is vertex-transitive, $\Gamma$ is not $(G,s)$-distance-transitive, and for some vertex-intransitive subgroup $N\lhd G$ the code $\C=\alpha^N$  is $(G_\C,s)$-neighbour-transitive -- thus giving a $(G/N,s)$-distance-transitive quotient graph $\Gamma_N$.
\end{problem}




\subsection{Further techniques for analysing codes}\label{sec:anal}

We now define a concept that has proved useful for keeping track of code invariants, especially for analysing $s$-neighbour-transitive codes.

\begin{definition}\label{GinvDef}
 Let $G$ be a group acting on a set $\Omega$ and let $\iota:\Omega\rightarrow S$, for some set $S$. If for all $\alpha\in\Omega$ and $g\in G$ the equality $\iota(\alpha)=\iota(\alpha^g)$ holds, then $\iota$ is called \emph{$G$-invariant}. The elements of $S$ are called \emph{types}, namely if $\alpha\in\Omega$, then we say that $\alpha$ has \emph{type} $\iota(\alpha)$.
\end{definition}

The archetypal example of a $G$-invariant map $\iota:\Omega\rightarrow S$ is to take $S$ as the set of $G$-orbits in $\Omega$ and define $\iota(\alpha)=\alpha^G$, for each $\alpha\in\Omega$. We consider another example below for the Johnson graphs which are defined in Definition~\ref{defJohnson}. Note that this example will be typical for us, in that $\Omega$ will often be the vertex set of a graph containing a code of interest. 

\begin{example}\label{exampleIntransJohnson}
 Let $\C$ be a $G$-neighbour-transitive code in the Johnson graph $J(v,k)$ (see Section~\ref{sec:JohnsonKneserPrelim}) with underlying set $\V$ such that $G$ fixes setwise some subset $U\subseteq \V$ with $0<|U|<v$. For a vertex $\alpha$ of $J(v,k)$ define $\iota(\alpha)=|\alpha\cap U|$. Let $g\in G$. Then, since $U^g=U$, it follows that $\iota(\alpha^g)=|\alpha^g\cap U|=|(\alpha\cap U)^g|=\iota(\alpha)$. Hence $\iota$ is $G$-invariant. 
\end{example}

In Example~\ref{exampleIntransJohnson}, since $G$ acts transitively on $\C$, every element of $\C$ has the same type. The next result expands on this observation.

\begin{lemma}\label{invlemma}
 Let $\C$ be a $(G,s)$-neighbour-transitive code with minimum distance $\delta$ in the connected graph $\Gamma=(V,E)$, let $\iota$ be a $G$-invariant map on $V$, let $\alpha\in\C$ and let $i\in\{0,1,\ldots,s\}$. Then the following hold.
 \begin{enumerate}[$(1)$]
  \item All vertices in $\C_i$ have the same type.
  \item $B_i(\alpha)$ contains at most $i+1$ different types of vertices.
  \item If $\delta\geq 2i$ then all vertices in $\Gamma_i(\alpha)$ have the same type. 
 \end{enumerate}
\end{lemma}


\begin{proof}
 Let  $\nu_1,\nu_2\in\C_i$. By Definition~\ref{defsneighbourtrans}, $G$ acts transitively on $\C_i$ and hence there exists $g\in G$ such that $\nu_1^g=\nu_2$. Since $\iota$ is $G$-invariant,  $\iota(\nu_1)=\iota(\nu_1^g)=\iota(\nu_2)$. This proves part (1). Next, $B_i(\alpha)\subseteq  \C\cup\C_1\cup\cdots\cup\C_i$. By part (1), the type of any vertex $\mu\in B_i(\alpha)$ is the same as the type of a vertex in $\C_j$, for some $j\in\{0,1,\ldots,i\}$. Hence there are at most $i+1$ possibilities for $\iota(\mu)$, and part (2) is proved. Finally, if $\beta\in\Gamma_i(\alpha)$ and $\delta\geq 2i$ then, by the triangle inequality, $d(\beta,\gamma)\geq i$ for all $\gamma\in \C\setminus\{\alpha\}$, and hence $\beta\in\C_i$. Thus $\Gamma_i(\alpha)\subseteq\C_i$, and so, by part (1), all vertices in $\Gamma_i(\alpha)$ have the same type, proving part (3).
\end{proof}

Lemma~\ref{invlemma} was used implicitly in \cite{liebler2014neighbour} to classify the codes in Example~\ref{exampleIntransJohnson}, that is, $G$-neighbour-transitive codes in $J(v,k)$ (see Section~\ref{sec:JohnsonKneserPrelim}) where $G$ acts intransitively on $\V$. We include a different proof of this result, below, in order to illustrate how Lemma~\ref{invlemma} may be applied. For a set $U$ and integer $j\leq |U|$, we denote by $\binom{U}{j}$ the set of all $j$-element subsets of $U$.

\begin{theorem}\label{thmIntransJohnson}\cite[Proposition~3.3]{liebler2014neighbour}
 Let $\C$ be a non-trivial $G$-neighbour-transitive code in $\Gamma=J(v,k)$, where $2\leq k<v$, and suppose that $U\subseteq\V$ is such that $G$ fixes $U$ setwise and $0<|U|<v$. Then $\C$ has minimum distance $1$ and, possibly replacing $U$ by $\V\setminus U$, one of the following holds.
 \begin{enumerate}[$(1)$]
     \item $k<|U|$ and $\C=\{\alpha\in V(\Gamma)\mid \alpha\subseteq U\}$;
     \item $k>|U|$ and $\C=\{\alpha\in V(\Gamma)\mid U\subseteq \alpha\}$.
 \end{enumerate}
\end{theorem}

\begin{proof}
 We consider the type $\iota(\alpha)=|\alpha\cap U|$ of a vertex $\alpha$ of $\Gamma$, as in Example~\ref{exampleIntransJohnson}.  By Lemma~\ref{invlemma}, all vertices in $\C$ have the same type, and all vertices in $\C_1$ have the same type. Moreover, for every $\alpha\in\C$, $\Gamma_1(\alpha)\subseteq \C\cup\C_1$, and hence either every vertex of $\Gamma_1(\alpha)$ has the same type, or $\delta(\C)=1$ and the vertices in $\Gamma_1(\alpha)$ have one of precisely two different types. Note that $\iota(\alpha)\leq |\alpha|=k$, and equality holds if and only if $\alpha\subseteq U$.

 First, suppose that $\iota(\alpha)=k$ for all $\alpha\in\C$. Then $\cup_{\alpha\in\C}\,\alpha\subseteq U$ and since $\C$ is non-trivial ({\em i.e.}, $|\C|\geq 2$) we have $|U|>k$. Let $\alpha\in\C$. Then $\alpha\subseteq U$ and $\Gamma_1(\alpha)$ consists of the set of all vertices $\nu\cup\{a\}$, where $\nu\in {\alpha\choose k-1}$ and $a\in\V\setminus \alpha$. Moreover, $\iota(\nu\cup\{a\})=k$ if $a\in U$ (there are $k(|U|-k)>0$ such pairs $(\nu,a)$) and $\iota(\nu\cup\{a\})=k-1$ if $a\notin U$ (there are $k(v-|U|)>0$ such pairs $(\nu,a)$). Hence $\Gamma_1(\alpha)$ contains precisely two types of vertices and so, by Lemma~\ref{invlemma}, $\delta(\C)=1$ and, for every codeword $\alpha\in \C$, each element of $\Gamma_1(\alpha)$ having type $k$ is again a codeword. Observe that the induced subgraph of $\Gamma$ on the set ${U\choose k}$  is isomorphic to $J(|U|,k)$ and is thus connected, and hence the set ${U\choose k}$ is precisely the set of vertices of type $k$; the previous sentence implies that ${U\choose k}\subseteq\C$. Hence $\C={U\choose k}$ and part (2) holds.

 Thus we may assume that codewords $\alpha$ have (constant) type $\iota(\alpha)=\ell$ strictly less than $k$. Suppose next that  $\ell=|U|$. This implies that $|U|<k$.  Let $\alpha\in\C$. Then $U\subset \alpha$ and $\Gamma_1(\alpha)$ consists of the set of all vertices $\nu\cup\{a\}$, where $\nu\in {\alpha\choose k-1}$ and $a\in\V\setminus \alpha$. Moreover, $\iota(\nu\cup\{a\})=\ell$ if $\nu\cap U=\alpha\cap U=U$ (there are $(k-\ell)(v-k)>0$ such pairs $(\nu,a)$) and $\iota(\nu\cup\{a\})=\ell-1$ if $\nu\cap U\neq\alpha\cap U$ (there are $\ell(v-k)>0$ such pairs $(\nu,a)$). Hence $\Gamma_1(\alpha)$ contains precisely two types of vertices, and by Lemma~\ref{invlemma}, each element of $\Gamma_1(\alpha)$ of type $\ell$ is again a codeword. Now the induced subgraph of $\Gamma$ on the set of vertices $\{\alpha\in V(\Gamma)\mid U\subseteq \alpha\}$ is isomorphic to $J(v-\ell,k-\ell)$ with vertex set ${{\V\setminus U}\choose {k-\ell}}$. Since $J(v-\ell,k-\ell)$ is connected, we deduce that $\{\alpha\in V(\Gamma)\mid U\subseteq \alpha\}\subseteq\C$ and in fact equality holds as in part (1).

 This leaves the case where $0\leq \ell< \min\{|U|,k\}$. 
If $\ell=0$, then all codewords $\alpha$ are contained in $\V\setminus U$, and replacing $U$ with $\V\setminus U$, the argument in the previous paragraph  shows that part (1) holds. Similarly, if $v-|U|=k-\ell$ then  all codewords $\alpha$ contain $\V\setminus U$, and replacing $U$ with $\V\setminus U$, the argument in paragraph two of the proof shows that part (2) holds. Thus we may assume in addition that $0<\ell$ and $v-|U|>k-\ell$. 
However in this case a codeword $\alpha$  is adjacent to vertices of types $\ell-1,\ell$ and $\ell+1$. This contradicts Lemma~\ref{invlemma}(2), completing the proof. 
\end{proof}

\section{Commentary on the origins of completely-transitive codes}\label{sec:ct}

In a 1987 preprint \cite[Section 7]{sole1987completely}, Patrick Sol\'{e} introduced the term `completely transitive code'  in the context of binary linear codes. His concept is different in a number of respects from that introduced in 
Definition~\ref{defsneighbourtrans}(3), and we comment more below. 
To the best of our knowledge, the notion of a `completely-transitive code in a graph' dates back to 1988, when Chris Godsil suggested to the second author the study of completely-transitive codes in Johnson graphs (see Definition~\ref{defJohnson}). Since a subset of vertices in a Johnson graph $J(v,k)$ is a set of $k$-subsets of the underlying $v$-set $\V$, each code $\C$ in $J(v,k)$ has a natural interpretation as a `design' with point-set $\V$ and with constant block size $k$. Thus completely-transitive codes in Johnson graphs were first called \emph{completely-transitive designs}. In 1988, Bill Martin was a PhD student of Chris Godsil studying completely regular designs (completely regular codes in $J(v,k)$) and Chris's hope was that by studying the more restricted family of completely transitive designs we might discover new examples, and characterisations. An account of this investigation, dating from 1997, can be found in~\cite{GodsilPraeger1997}. It was never published as a journal article, but was  posted on the arXiv in 2014 because of repeated requests for copies.

When regarded as codes in the Hamming graph $H(n,2)$ (Definition~\ref{defHamming}), the `completely-transitive codes'  defined by Sol\'{e}~ \cite[Section 7]{sole1987completely}, or see \cite{sole1990completely} for the 1990 published version, are precisely those binary linear codes $\C\subseteq \mathbb{F}_2^n$ that are $G$-completely-transitive in the sense of Definition~\ref{defsneighbourtrans}(3) for the subgroup $G=T\rtimes G_0$ of affine transformations, where $T$ is the group of translations of $\mathbb{F}_2^n$ by codewords in $\C$, and $G_0$ is the subgroup of permutation matrices which leave $\C$ invariant.
Note that $G_0$ is the group traditionally regarded as the automorphism group of a linear code. In ~\cite[Section 7]{sole1990completely}, Sol\'{e} gave examples of such codes, and also a  necessary condition,  and a  sufficient condition (separate conditions), for complete transitivity in terms of the natural action of $G_0$ as a permutation group on a set of $n$ points. Shortly afterwards, in 1991, Rif\`{a} and Pujol \cite{rifapujol1991} (or see \cite[Proposition 13]{borges2019completely}) showed that, for a binary linear completely transitive code $\C$ of length $n$, the natural coset graph on the set $\mathbb{F}_2^n/\C$ of additive cosets of $\C$ in $\mathbb{F}_2^n$ is distance-transitive. Sol\'{e}'s  concept of complete transitivity extends naturally to linear codes in $H(n,q)$ for arbitrary prime powers $q$, as does the link with distance transitivity of the quotient graph $\mathbb{F}_q^n/\C$ (see Proposition~\ref{p:sdt} with $s=\rho(\C)$). For this reason, in order to distinguish Sol\'{e}'s completely transitive linear codes from the more general family of completely transitive codes in Hamming graphs, Sol\'{e}'s notion is sometimes called    
\emph{coset-complete-transitivity}, see \cite{michaelmast,Giudici1999647}.
For linear codes in $H(n,q)$, where $q$ is a prime power, the concepts of complete transitivity and coset-complete transitivity are equivalent if  $q\leq 3$ and definitely not equivalent if $q=7$ or $q\geq 9$, see \cite[Theorem 1.3 and Example 3.1]{Giudici1999647}.


Giudici, in his masters thesis~\cite{michaelmast} and a joint paper with the second author~\cite{Giudici1999647}, introduced the notion given in Definition~\ref{defsneighbourtrans} of a $G$-completely-transitive code in a Hamming graph $H(n,q)$ for arbitrary $n$ and $q$, and made a general study of the structure of such graphs. He gave families of examples, and significant structural descriptions,  which pointed to the importance of those codes in $H(n,q)$ for which the    automorphism group $G$ induces a transitive action on the set of $n$ entries of codewords, see \cite[Sections 4, 5 and 6]{Giudici1999647}. A recent survey article of Borges, Rifa and Zinoviev contains an overview of this and more recent work, see in particular \cite[Section 3] {borges2019completely} on  coset-completely transitive codes. 

The error-correcting capacity $e$ of a binary coset-completely transitive code was shown by Borges, Rif\`{a} and Zinoviev to be no larger than three,  \cite{borges2000nonexistence, borges2001nonexistence} or see  \cite[Theorem 12] {borges2019completely}, 
but there are families of examples with unbounded covering radius (with $e=1$) \cite{rifazinoviev2009}. We will see that similar statements can be made about the more general family of $s$-neighbour-transitive codes. In particular we will prove the following result in Section~\ref{sec:HammingNonexistence}.

\begin{theorem}\label{thm:CTham}
 Suppose that $\C$ is an $s$-neighbour-transitive code in the Hamming graph $H(n,q)$ such that $\C$ has error-correcting capacity $e$ with $\min\{e,s\}\geq 4$. Then $n \geq 9$, $q=2$, $e=\lfloor \frac{n-1}{2}\rfloor$, $s=\lceil \frac{n-1}{2}\rceil$, and $\C$ is equivalent to the binary repetition code $\Rep_n(2)$ (see Definition~$\ref{def:productandrepetition}(3)$).  
\end{theorem}

\noindent
Borges, Rif\`{a} and Zinoviev  showed further that coset-completely transitive codes may be obtained via some innovative combinatorial methods, such as a Kronecker product construction (\cite{rifazinoviev2010} or see \cite[Section 5.5]{borges2019completely}), or a lifting procedure (\cite{rifazinoviev2011} or see \cite[Section 5.2]{borges2019completely}), while Gill, Gillespie and Semeraro~\cite[Theorem C]{gill2018conway} constructed new families from the incidence matrices of certain designs. Also, sometimes a well-known family of codes may contain only a few codes which are completely transitive. For example, a Preparata code of length $n$, or its extension, is completely transitive if and only if $n=15$,  (\cite{gillespie2012nord} or see \cite[Theorem 15]{borges2019completely}); these exceptional completely transitive codes are the Nordstrom--Robinson codes. 

In this chapter we will see that a similar situation arises for the larger families of neighbour-transitive, and 2-neighbour-transitive codes, where substructures in graphs, and in finite geometries such as generalised quadrangles, have been used in constructions, and where group theory, sometimes relying on the finite simple group classification, has been applied to achieve classifications.

\section{Graphs of interest}\label{sec:graphs}


In this section we introduce several families of graphs. For each family we provide at least one example of a code arising from an interesting combinatorial or algebraic structure. Most of these families are analysed further in later sections. We have already informally met the first family of graphs, the Hamming graphs. They will also receive the most attention throughout the chapter.

\subsection{Hamming graphs}\label{sec:HammingPrelim}

\begin{definition}\label{defHamming}
 Let $\N$ be a set of size $n$ and $\Q$ a set of size $q$, where $n,q\geq 2$. We define the \emph{Hamming graph} $H(n,q)$ in the following two equivalent, but dual, ways:
 \begin{enumerate}[(1)]
  \item The vertex set of $H(n,q)$ is the set of all $n$-tuples $(a_1,\ldots,a_n)$, where $a_i\in \Q$ for each $i\in \N$ and we have identified $\N$ with $\{1,\ldots,n\}$. Two such $n$-tuples form an edge if and only if they differ in precisely one position.
  \item The vertex set of $H(n,q)$ is the set of all functions $\alpha:\N\rightarrow\Q$ with an edge between vertices $\alpha$ and $\beta$ when there exists a unique $i\in\N$ such that $\alpha(i)\neq\beta(i)$. 
 \end{enumerate}
 The set $\N$ is called the \emph{set of entries}, and the set $\Q$ the \emph{alphabet}, of $H(n,q)$. We may use the notation $H(\N,q)$, $H(n,\Q)$ or $H(\N,\Q)$ if we wish to indicate either or both of the sets $\N$ or $\Q$ explicitly.
\end{definition}

To see that the two definitions given above are equivalent: let $\N=\{1,\ldots,n\}$ and consider functions $\alpha,\beta:\N\rightarrow\Q$. The $n$-tuples $(\alpha(1),\ldots,\alpha(n))$ and $(\beta(1),\ldots,\beta(n))$ differ in precisely one entry if and only if there exists a unique $i\in \N$ such that $\alpha(i)\neq\beta(i)$. We will use whichever is the most convenient notation for each particular context, functions or $n$-tuples. The following example illustrates this further.

\begin{example}\label{exam:smallHamming}
 Let $\N=\F_2^3$, let $\Q=\F_2$, let $\F_2[x_1,x_2,x_3]$ denote the polynomial ring over $\F_2$ in the three variables $x_1,x_2,x_3$, and let $\C$ be the $\F_2$-vector space
 \[
  \C=\la 1,x_1,x_2,x_3\ra\subseteq \F_2[x_1,x_2,x_3].
 \]
 Considering each polynomial in $\C$ as a function from $\N$ to $\Q$, the set $\C$ defines a linear code in $H(8,2)$. The elements $0$ and $1$ in $\C$ correspond to the $8$-tuples $(0,\ldots,0)$ and $(1,\ldots,1)$, respectively. The remaining elements of $\C$ are precisely the linear polynomials in $\F_2[x_1,x_2,x_3]$ and, upon evaluation on the elements of $\N$, these correspond to the characteristic vectors of the $2$-flats (the affine $2$-spaces) of the affine geometry $\AG_3(2)$. If the entries are labelled by the elements of $\N$ as follows
 \[  (0,\,e_1,\,e_2,\,e_1+e_2,\,e_3,\,e_1+e_3,\,e_2+e_3,\,e_1+e_2+e_3),
 \]
 then we find that $\C$ consists of the following $8$-tuples:
 \[
  \begin{array}{cc}
      (1,1,1,1,1,1,1,1) & (0,0,0,0,0,0,0,0) \\
      (1,1,1,1,0,0,0,0) & (0,0,0,0,1,1,1,1) \\
      (1,1,0,0,1,1,0,0) & (0,0,1,1,0,0,1,1) \\
      (1,1,0,0,0,0,1,1) & (0,0,1,1,1,1,0,0) \\
      (1,0,1,0,1,0,1,0) & (0,1,0,1,0,1,0,1) \\
      (1,0,1,0,0,1,0,1) & (0,1,0,1,1,0,1,0) \\
      (1,0,0,1,1,0,0,1) & (0,1,1,0,0,1,1,0) \\
      (1,0,0,1,0,1,1,0) & (0,1,1,0,1,0,0,1).
  \end{array}
 \]
 Note that $\C$ is an extended Hamming code of length $8$ \cite[Example, p.~27]{macwilliams1978theory}, but is also known as the Reed--Muller code $\RM_2(1,3)$ (see Definition~\ref{def:genReedMuller}). 
\end{example}


The full automorphism group of the Hamming graph $\Gamma=H(n,q)$ is the semi-direct product $\Aut(\Gamma)=B\rtimes L$, where the \emph{base group} $B$ is $\Sym(\Q)^n$ and the \emph{top group} $L$ is $\Sym(\N)$ (see \cite[Theorem 9.2.1]{brouwer}). Let $x=h\sigma\in\Aut(\Gamma)$, where $h=(h_1,\ldots,h_n)\in B$ and $\sigma\in L$. Then $h, \sigma$ and $x$ act on an $n$-tuple $\alpha=(a_1,\ldots,a_n)$ via
\begin{equation}\label{e:hamaut}
 \alpha^h= \left(a_1^{h_1},\ldots,a_n^{h_n}\right), \quad  
 \alpha^\sigma= \left(a_{\left(1^{\sigma^{-1}}\right)},\ldots,a_{\left(n^{\sigma^{-1}}\right)}\right), \quad 
 \text{and}\quad \alpha^x=\left(a_{1{\sigma^{-1}}}^{h_{1{\sigma^{-1}}}},\ldots,a_{n{\sigma^{-1}}}^{h_{n{\sigma^{-1}}}}\right).
\end{equation}
For example, $(a_1,a_2,a_3,a_4)^{(1\,2\,3\,)}=(a_3,a_1,a_2,a_4)$. If instead we consider a vertex $\alpha$ to be a function $\N\rightarrow\Q$ then $h, \sigma$ and $x$ act on $\alpha$ via
\begin{equation}\label{e:hamaut2}
\alpha^h(i)=\left(\alpha(i)\right)^{h_i},\quad \alpha^\sigma(i)=\alpha\left(i^{\sigma^{-1}}\right),  \quad \text{and}\quad  \alpha^x(i)=\left( \alpha(i^{\sigma^{-1}})\right)^{h_{i\sigma^{-1}}}.
\end{equation}

Let $\C$ be a code in $H(n,q)$, $G\leq \Aut(\C)$ and let $\M$ be a subset of $\N$ with $m=|\M|$. Then we define
\begin{enumerate}[(a)]
    \item the \emph{projection} $\pi_\M(\C)$ of $\C$ with respect to $\M$ to be the code in $H(\M,q)$ consisting of the functions $\alpha|_\M:\M\rightarrow\Q$ given by restricting $\alpha$ to $\M$, for each $\alpha\in \C$. Moreover, we define 
    \item $\chi_\M(G)$ to be the subgroup of $\Aut(H(\M,q))$ induced by the setwise stabiliser $G_{\M}$ in its action on the restrictions $\alpha|_\M$ to $\M$ of functions $\N\rightarrow \Q$. 
    \item We say that  a code $\C$ is \emph{linear}, or \emph{$\F_q$-linear}, if $\Q=\F_q$ and $\C$ is an $\F_q$-vector subspace of $\F_q^n$. 
    \item If $\Q=\F_q$ then it makes sense to define the translation $t_\alpha:\beta\mapsto\alpha+\beta$ by a vertex $\alpha$ of $H(n,q)$, and if $\C$ is linear then $\Aut(\C)$ contains the subgroup $T_{\C}=\{t_\alpha\mid\alpha\in \C\}$. If $0$ is a distinguished element of $\Q$ and $\alpha$ is a vertex in $H(n,q)$ then $\supp(\alpha)=\{i\in\N\mid \alpha(i)\neq 0\}$ and $\wt(\alpha)=|\supp(\alpha)|$.
\end{enumerate}

\begin{example}\label{exam:smallHammingAutos}
 Consider the code $\C=\la 1,x_1,x_2,x_3\ra$ (an $\F_2$ vector subspace) in $H(8,2)$ with $\N=\F_2^3$ and $\Q=\F_2$, as in Example~\ref{exam:smallHamming}. Then for each $\alpha\in \C$, $\Aut(\C)$ contains the translation $t_\alpha:\beta\rightarrow \alpha+\beta$, and moreover each $\sigma\in\AGL_3(2)\leq \Sym(\N)$ defines an automorphism of $\C$. Hence $\Aut(\C)$ contains $G=T_{\C}.\AGL_3(2)$. Let $\M=\la e_1,e_2\ra$, where $e_i$ corresponds to $x_i$ in Example~\ref{exam:smallHamming}. Then the projection code $\C'=\pi_{\M}(\C)=\la 1,x_1,x_2\ra$ and we may express the codewords of $\pi_{\M}(\C)$ as $4$-tuples by taking the first four entries of each $8$-tuple from Example~\ref{exam:smallHamming}:
 \[
  \begin{array}{cc}
      (1,1,1,1) & (0,0,0,0) \\
      (1,1,0,0) & (0,0,1,1) \\
      (1,0,1,0) & (0,1,0,1) \\
      (1,0,0,1) & (0,1,1,0).
  \end{array}
 \]
 Also, the projection $\chi_{\M}(G)=T_{\C'}.\AGL_2(2)$ is a subgroup of $\Aut(\C')$.
\end{example}

\begin{definition}\label{def:productandrepetition}
 Let $\C$ be a code in $H(\N,\Q)$ and let $\M=\N\times\{1,\ldots,k\}$, from which we introduce the following definitions related to $H(\M,\Q)$.
 \begin{enumerate}[(1)]
  \item If $\alpha_1,\ldots,\alpha_k$ are vertices of $H(\N,\Q)$ (that is, functions $\N\rightarrow \Q$) then we define  a vertex of $H(\M,\Q)$, $\beta_{\alpha_{1},\ldots,\alpha_k}:\M\rightarrow \Q$, by
  \[
   \beta_{\alpha_{1},\ldots,\alpha_k}:(i,j)\mapsto \alpha_j(i).
  \]
  \item The \emph{$k$-fold product} $\Prod_k(\C)$ is the code in $H(\M,\Q)$ given by
  \[
   \Prod_k(\C)=\{\beta_{\alpha_{1},\ldots,\alpha_k}\mid \alpha_1,\ldots,\alpha_k\in\C\}.
  \] 
  \item The \emph{$k$-fold repetition} $\Rep_k(\C)$ in $H(\M,\Q)$ is
  \[
   \Rep_k(\C)=\{\beta_{\alpha,\ldots,\alpha}\mid \alpha\in\C\};
  \]
  and if $|\N|=1$ so $H(\N,\Q)$ is degenerate (a complete graph on $q$ vertices), then this code is the usual repetition code, and we write $\Rep_k(q)$ if $\C=H(1,\Q)$.
  \item For $g_1,\ldots,g_k\in\Sym(\N)$, we define $\sigma_{g_1,\ldots,g_k}$ to be the automorphism of $H(\M,\Q)$ that maps the vertex 
  \[
   \gamma:(i,j)\mapsto \gamma(i,j)\quad \text{to}\quad \gamma^{\sigma_{g_1,\ldots,g_k}}:(i,j)\mapsto \gamma\Big(i^{g_j^{-1}},j\Big).
  \]
  \item For $h_1,\ldots,h_k\in\Sym(\Q)$, we define $x_{h_1,\ldots,h_k}$ to be the automorphism of $H(\M,\Q)$ that maps the vertex 
  \[
   \gamma:(i,j)\mapsto \gamma(i,j)\quad \text{to}\quad \gamma^{x_{h_1,\ldots,h_k}}:(i,j)\mapsto (\gamma(i,j))^{h_j}.
  \]
 \end{enumerate}
  Note that the automorphism group of $\Prod_k(\C)$ contains the wreath product $\Aut(\C)\wr\s_k$ and the automorphism group of $\Rep_k(\C)$ contains $\Aut(\C)\times\s_k$.
\end{definition}

\subsection{Johnson and Kneser graphs}\label{sec:JohnsonKneserPrelim}

Recall that $\V\choose k$ denotes the set of all $k$-element subsets of a set $\V$; this set forms the vertex set of both the Johnson and Kneser graphs


\begin{definition}\label{defJohnson}
 Let $\V$ be a set of size $v$ and let $k$ be an integer with $2\leq k\leq v-1$. The \emph{Johnson graph} $J(v,k)$ has vertex set $\V\choose k$ and a pair of distinct vertices are defined to be adjacent if they intersect in $(k-1)$-subset. We say that $\V$ is the \emph{underlying set} of $J(v,k)$.
\end{definition}

If $v\neq 2k$ then $\Aut(J(v,k))=\Sym(\V)$ and if $v=2k$ then $\Aut(J(v,k))=\Sym(\V)\times C_2$ \cite[Theorem~9.1.2]{brouwer}. As mentioned in the introduction, we refer the reader to \cite{praeger2021codes} for a recent survey regarding neighbour-transitive codes in Johnson graphs.

\begin{definition}\label{defKneser}
 Let $\V$ be a set of size $v$ and let $k$ be an integer with $2\leq k\leq (v-1)/2$. The \emph{Kneser graph} $K(v,k)$ has vertex set $\V\choose k$ and a pair of distinct vertices are defined to be adjacent if they are disjoint. If $v=2k+1$ then $K(v,k)$ is called the \emph{odd graph} $O_{k+1}$. We say that $\V$ is the \emph{underlying set} of $K(v,k)$ (or $O_{k+1}$).
\end{definition}

The automorphism group of $K(v,k)$ is $\Sym(\V)$ \cite[Corollary~7.8.2]{godsil2013algebraic}. The next example explores the relationship between codes in Johnson graphs and codes in Kneser graphs.

\begin{example}\label{ex:jk}
 Let $\V$ be a set of size $v$, let $k$ satisfy $2\leq k\leq (v-1)/2$, and let $J(v,k)$ and $K(v,k)$ be the Johnson and Kneser graphs, respectively, with underlying set $\V$. 
 Since both graphs have the same vertex set, namely ${\V\choose k}$, a code in one is also a code in the other, generally with different parameters. For instance:
 \begin{enumerate}
  \item[(a)] Let $U\subseteq \V$ such that $k< |U|< v$ and let $\C=\{\alpha\in{\V\choose k}\mid \alpha\subset U\}$. By Theorem~\ref{thmIntransJohnson}(1), as a code in the Johnson graph, $\C$ has minimum distance $1$ and is neighbour-transitive. We now consider $\C$ to be a code in the Kneser graph. In this case, since there exists a pair of disjoint $k$-subsets of $U$ if and only if $|U|\geq 2k$, it follows that $\C$ has minimum distance $1$ in the Kneser graph if and only if $|U|\geq 2k$. Moreover, since there always exists a pair of $k$-subsets of $U$ at distance $2$ in $K(v,k)$ (in particular, $|U|\geq k+1$ implies there exists such a pair intersecting in a $(k-1)$-subset of $U$), we deduce that $\C$ has minimum distance $2$ in the Kneser graph when $|U|<2k$. Moreover, it turns out that $\C$ is neighbour-transitive in the Kneser graph if and only if $|U|=v-1$ (see Theorem~\ref{thm:KneserIntrans}(1) and Example~\ref{ex:KneserIntrans}).
  \item[(b)] Let $U\subseteq \V$ such that $0< |U|< k$ and let $\C=\{\alpha\in{\V\choose k}\mid U\subset \alpha\}$. Theorem~\ref{thmIntransJohnson}(2) again tells us that, as a code in the Johnson graph, $\C$ has minimum distance $1$ and is neighbour-transitive. However, since every codeword contains $U$,  no pair of elements of $\C$ are disjoint. Hence $\C$ has minimum distance at least $2$ in the Kneser graph, and it is not difficult to see that the minimum distance equals $2$ since $v\geq 2k+1$. Moreover, it turns out that $\C$ is always neighbour-transitive in the Kneser graph (see Theorem~\ref{thm:KneserIntrans}(3) and Example~\ref{ex:KneserIntrans}).
 \end{enumerate}
  
\end{example}

\subsection{$q$-Analogues of various graphs}\label{sec:qAnaloguesPrelim}

Roughly speaking, if a combinatorial definition is phrased in terms of sets then, for a prime power $q$, the \emph{$q$-analogue} of this definition is obtained by rephrasing the original definition in terms of $\F_q$-vector spaces. The next definition gives the $q$-analogues of the Hamming graphs. Codes in these graphs are known as \emph{rank-metric} codes, and have received interest lately due to their application in network coding, see \cite{bartz2022rankmetric} for a recent survey.

\begin{definition}\label{defBilinearForms}
 Let $X\cong\F_q^m$ and $Y\cong\F_q^n$ with $m,n\geq 2$. The \emph{bilinear forms graph} $H_q(m,n)$ is the graph with vertex set the vector space of all linear maps $\alpha:X\rightarrow Y$ where two vertices $\alpha$ and $\beta$ are adjacent if $\rank(\alpha-\beta)=1$.
\end{definition}


\begin{example}
 Delsarte \cite{delsarte1978bilinear} and Gabidulin \cite{gabidulin1985} independently introduced the following class of codes. Let $n,k,s$ be positive integers such that $\gcd(n,s)=1$. Identify $X$ and $Y$ with $\F_{q^n}$, considered as an $n$-dimensional $\F_q$-vector space, and note that each of the polynomials $x^{q^{is}}$ defines an $\F_q$-linear map $X\to Y$.  
 
 Define $\C$ 
 to be the code in $\Gamma=H_q(n,n)$ given by the polynomials in the $\F_{q^n}$-vector space 
 \[
 \C= \left\langle x,x^{q^s},\ldots,x^{q^{s(k-1)}}\right\rangle_{\F_{q^n}}.
 \]
 Note that each polynomial in $\C$ is a \emph{linearised} polynomial (see \cite[Section~3.4]{lidl1997finite}) and thus defines an $\F_q$-linear map from $X$ to $Y$. Also, $\C$ has minimum distance $\delta=n-k+1$. Hence, as long as $k\leq n-1$, the set $\Gamma_1(0)$ comprises all rank $1$ linear maps. The trace function $\tr(x)=x^{q^{n-1}}+\cdots+x^q+x$ is one such rank $1$ linear map and all others may be written as $f_{a,b}(x)=a\cdot\tr(bx)$ for some $a,b\in\F_{q^n}^\times$. Since $\C$ is an $\F_{q^n}$-vector space, it follows that $\Aut(\C)$ contains all translations by elements of $\C$. Moreover, $\Aut(\C)$ contains a subgroup isomorphic to $(\F_{q^n}^\times \times \F_{q^n}^\times)/\F_q^\times$ given by the maps $f(x)\mapsto a\cdot f(bx)$ for each $a,b\in\F_{q^n}^\times$. Thus $\Aut(\C)$ acts transitively on $\C$ and the stabiliser in $\Aut(\C)$ of the zero map acts transitively on the set of all rank $1$ linear maps. Hence, by Lemma~\ref{lem:largeDeltaDistTrans}, $\C$ is neighbour-transitive. 
\end{example}

The above example suggests the following open problem.

\begin{problem}\label{prob:bilinearForms}
 Find more examples and characterise families of $s$-neighbour-transitive codes in $H_q(m,n)$.
\end{problem}

Next we introduce the $q$-analogues of the Johnson graphs. For $X=\F_q^d$, we write ${X\choose k}_q$ for the set of $k$-dimensional subspaces of $X$.

\begin{definition}\label{defGrassmann}
 Let $X=\F_q^d$. The \emph{Grassmann graph} $J_q(d,k)$ has vertex set ${X\choose k}_q$ and vertices $\alpha$ and $\beta$ are adjacent when $\alpha\cap\beta$ has dimension $k-1$.
\end{definition}

For an integer $k\geq 2$, a \emph{$k$-spread} of a vector space $V$ is set of $k$-dimensional subspaces of $V$ such that each non-zero vector of $V$ is contained in precisely one spread element. In the setting of projective geometry, a spread is a partition of the points of $\PG_{d-1}(q)$ such that each part forms a $(k-1)$-dimensional projective subspace. A \emph{regulus} of $\PG_{2k-1}(q)$ is a set $\R$ of size $q+1$ consisting of pairwise disjoint $(k-1)$-dimensional projective subspaces such that any line meeting at least $3$ elements of $\R$ meets every element of $\R$; see \cite[Section~5.1]{dembowskiFiniteGeometries}. Three pairwise disjoint $(k-1)$-dimensional projective subspaces contained in a $(2k-1)$-subspace determine a unique regulus containing all three of them. A spread $\S$ is called \emph{regular} (or sometimes \emph{Desarguesian}) if $\S$ contains all the elements of the regulus determined by any three pairwise distinct elements of $\S$ contained in a projective $(2k-1)$-subspace. Equivalently, a regular spread is one obtained via \emph{field reduction}, see \cite{lavrauw2015field}. The next example uses field reduction to show that a regular spread is a neighbour-transitive code in a Grassmann graph.

\begin{example}\label{exRegularSpreadGrassmann}
 Let $k$ be a positive integer dividing $d$, let $Y=\F_{q^k}^{d/k}$ and $X=\F_q^d$. Then $X$ and $Y$ are isomorphic as $\F_q$-vector spaces. Fix such an $\F_q$-vector space isomorphism $\phi:Y\rightarrow X$. Let $\C$ be the image under $\phi$ of the set of all $1$-dimensional $\F_{q^k}$-subspaces of $Y$. Then $\C$ is a regular spread of $X$ and $\C$ is a code in $J_q(d,k)$ with minimum distance $\delta=k$. The stabiliser of $\C$ in $\GaL_d(q)$ is $\GaL_{d/k}(q^k)$, and hence $\Aut(\C)\cong \GaL_{d/k}(q^k)/Z$, where $Z$ is the center of $\GL_d(q)$. Thus $\Aut(\C)$ acts transitively on $\C$. Moreover, if $\alpha\in \C$ then $\Aut(\C)_\alpha$ contains $(q^{d-k}:((q^k-1)\times\GL_{d/k-1}(q^k)))/Z$, which acts transitively on the set of $k$-dimensional $\F_q$-subspaces of $X$ that intersect $\alpha$ in a $(k-1)$-dimensional $\F_q$-subspace. Hence, by Lemma~\ref{lem:largeDeltaDistTrans}(2), $\C$ is neighbour-transitive.
\end{example}

\subsection{Two graphs related to incidence structures}\label{sec:incidenceGraphsPrelim}

An \emph{incidence structure} (see \cite{dembowskiFiniteGeometries}) is a triple $\S=(\P,\L,\I)$, where $\P$ and $\L$ are disjoint non-empty sets of elements called \emph{points} and \emph{lines}, respectively, and $\I\subseteq \P\times \L$ is a point-line incidence relation. If $(p,\ell)\in \I$ then we say that $p$ \emph{lies on} $\ell$, $p$ is \emph{incident} with $\ell$, or $\ell$ is \emph{incident} with $p$. The elements of $\I$ are called \emph{flags} and if $p\in \P$, $\ell\in\L$ and $(p,\ell)\notin\I$ then $(p,\ell)$ is called an \emph{antiflag}. If there exists an $\ell\in\L$ such that $(p_1,\ell),(p_2,\ell)\in\I$ then $p_1$ and $p_2$ are \emph{collinear}. The \emph{dual} incidence structure $\S^D=(\L,\P,\I^D)$ is obtained from $\S$ by interchanging the sets of points and lines, with incidence $\I^D$ defined by $(\ell,p)\in \I^D$ if $(p,\ell)\in \I$. 


\begin{definition}\label{defIncidenceStrucGraphs}
 Let $\S=(\P,\L,\I)$ be an incidence structure. Then the \emph{incidence graph} of $\S$ is the graph with vertex set $\P\cup\L$ and an edge $\{p,\ell\}$ whenever $(p,\ell)\in\I$ (and no further edges). The \emph{collinearity graph} of $\S$ is the graph with vertex set $\P$ and an edge $\{p_1,p_2\}$ whenever $p_1$ and $p_2$ are collinear.
\end{definition}

\begin{lemma}\label{thmIncidenceToCollinearity}
 Let $\Gamma$ be the incidence graph of an incidence structure $\S=(\P,\L,\I)$ and let $\C$ be an $s$-neighbour-transitive code in $\Gamma$ such that $\C\subseteq\P$. Then $\C$ is an $\lfloor s/2\rfloor$-neighbour-transitive code in the collinearity graph of $\S$.
\end{lemma}

\begin{proof}
 This follows from applying Lemma~\ref{lem:largeDeltaDistTrans} to $\C$ as a vertex subset of both the incidence graph of $\S$ and the collinearity graph of $\S$.
\end{proof}

\begin{example}\label{exam:pg3q}
 Let $\S=(\P,\L,\I)$ be the incidence structure obtained by letting $\P$ be the set of points, and $\L$ the set of lines, of the projective geometry $\PG_3(q)$, with the usual point-line incidence. Let $\Gamma$ be the incidence graph of $\S$, and let $\C$ be a regular spread of $\PG_3(q)$, as defined in Example~\ref{exRegularSpreadGrassmann} (with $d=4$, $k=2$), but considered now as a code in the incidence graph $\Gamma$. 
 \begin{enumerate}[$(a)$]
     \item  Then as a code in $\Gamma$,  $\C$ is completely transitive (and $2$-neighbour-transitive) with minimum distance $4$ and covering radius $2$; and 
     \item  as a code in the collinearity graph of $\S^D$, which is isomorphic to $J_q(4,2)$,  $\C$ is completely transitive (and neighbour-transitive) with minimum distance $2$ and covering radius $1$.
 \end{enumerate} 
\end{example}

\begin{lemma}\label{lem:pgeq}
    All the claims of Example~\ref{exam:pg3q} are valid.
\end{lemma}

\begin{proof}
    Let $X=\F_q^4$. It follows from Definition~\ref{defGrassmann} that the collinearity graph of $\S^D$ is $J_q(4,2)$, with vertex set $\binom{X}{2}_q$.  Also, from Example~\ref{exRegularSpreadGrassmann}, we see that $\Aut(\C)$ contains $\GL_2(q^2)/Z$, where $Z$ is the scalar subgroup of $\GL_4(q)$, and as a code in $J_q(4,2)$, $\C$ has minimum distance $2$, covering radius $1$ and $\Aut(\C)$ is transitive on both $\C$ and $\L\setminus\C$. Thus all the assertions of part (b) are valid. Now we consider $\C$ as a code in the incidence graph $\Gamma$ of $\S$. The set of code neighbours comprises the whole of the point set $\P$, since from the definition of a spread each projective point lies on exactly one line in the spread. Thus $\C$ has covering radius $2$. Also the subgroup $\GL_2(q^2)/Z$ of $\Aut(\C)$ is transitive on $\P$, and it follows that    $\C$ is $2$-neighbour-transitive and hence also completely transitive. Finally, any two codewords, that is lines $\ell, \ell'$ in $\C$, have no point in common and so are at distance strictly greater than $2$ in $\Gamma$. Since $\Gamma$ is bipartite this means that $\ell, \ell'$ are at distance at least $4$. In fact their distance is equal to $4$, since choosing points $p, p'$ on $\ell, \ell'$, respectively, and letting $\ell''$ be the line containing $p$ and $p'$, we have a path $(\ell, p, \ell'', p', \ell')$ of length $4$ in $\Gamma$. Thus all the assertions of part (a) are also valid.
\end{proof}

\section{Codes in Hamming graphs}\label{sec:hamming}

The next three propositions are applications of Lemma~\ref{lem:largeDeltaDistTrans} and are the starting point for our analysis of $s$-neighbour-transitive codes in Hamming graphs. We include here a proof of the first in order to illustrate the general idea. Recall the notions of $i$-homogeneous and $i$-transitive group actions from Section~\ref{sec:sym}.

\begin{proposition}\label{prop:HammingiHomogeneous}\cite[Proposition~2.5]{ef2nt}
 Let $\C$ be a $(G,s)$-neighbour-transitive code, with error-correction capacity $e\geq1$, in $\Gamma=H(n,q)=H(\N,q)$ and let $\alpha\in \C$. Then, for each $i\leq\min\{e,s\}$, $G_\alpha$ acts $i$-homogeneously on $\N$.
\end{proposition}

\begin{proof}
 Since the automorphism group of the Hamming graph is vertex-transitive, we may assume (replacing $\C$ if necessary by $\C^g$, for some $g\in \Aut(H(n,q))=\Sym(q)\wr\Sym(n)$) that $\alpha=(0,\ldots,0)\in\C$ for some distinguished element $0\in\Q$.  By Lemma~\ref{lem:largeDeltaDistTrans}, the stabiliser $G_\alpha$ acts transitively on $\Gamma_i(\alpha)$ for each $i\leq\min\{e,s\}$.  Now $\Gamma_i(\alpha)$ is the set of all weight $i$ vertices of $H(n,q)$. Let $I,J\in {\N\choose i}$ and $\beta_1,\beta_2\in \Gamma_1(\alpha)$ such that $\supp(\beta_1)=I$ and $\supp(\beta_2)=J$. Then there exists $g\in G$ such that $\beta_1^g=\beta_2$. Since $i\leq e$, the codeword $\alpha$ is the only vertex of $\C$ that has distance $i$ from $\beta_1$ and $\beta_2$, and thus $g\in G_\alpha$. This means that $g$ is of the form $g=(h_1,\ldots,h_n)\sigma$, where $\sigma\in L$, $(h_1,\ldots,h_n)\in B$, and  $0^{h_i}=0$ for each $i=1,\ldots,n$. It follows that $I^\sigma=J$, and the result is proved.
\end{proof}

Proposition~\ref{prop:HammingiHomogeneous} implies in particular that, for each $G$-neighbour-transitive code $\C$ in $H(\N,q)$ with error correction capacity at least $1$, the group $G$ induces a transitive action on $\N$. This action may be imprimitive, that is to say, $G$ may leave invariant some non-trivial partition of $\N$. It turns out that, in this case, the projection   $\pi_J(\C)$, for any part $J$ of such a partition, is also a neighbour-transitive code in the smaller Hamming graph $H(J,q)$. (Recall the projection maps $\pi_J$ and $\chi_J$ from Section~\ref{sec:graphs}.)

\begin{proposition}\label{prop:NTimprimitiveAction}
 Suppose that $\C$ is a $G$-neighbour-transitive code in $H(\N,q)$ with minimum distance $\delta\geq 3$, and suppose further that $\J$ is a non-trivial $G$-invariant partition of $\N$, and $J\in \J$ is such that $\pi_J(\C)$ is not the complete code $H(J,q)$. Then the following hold:
 \begin{enumerate}[$(1)$]
  \item  $\pi_J(\C)$ is $\chi_J(G)$-neighbour-transitive (\cite[Proposition~3.4]{gillespieCharNT}). 
  \item $\pi_J(\C)$ has minimum distance at least $2$ (\cite[Corollary~3.7]{gillespieCharNT}).
 \end{enumerate}
\end{proposition}

Now $\Aut(H(\N,\Q))=\Sym(\Q)\wr\Sym(\N)=B\rtimes L$. For a subgroup $G\leq B\rtimes L$, and $i\in \N$, let $G_i$ denote the subgroup consisting of all elements $x=h\sigma\in B\rtimes L$ lying in $G$ such that $i^\sigma = i$. By \eqref{e:hamaut}, such an element $x=h\sigma$ maps each vertex $(\alpha_1,\dots, \alpha_n)$ to a tuple with $i^{th}$ entry $\alpha_i^{h_i}$. Thus $G_i$ induces an action on the set $\Q_i$ of $i^{th}$ entries of vertices of $H(\N,\Q)$. (Of course $\Q_i$ is a copy of $\Q$.) A more intricate argument than that in the proof of Proposition~\ref{prop:HammingiHomogeneous} yields the following.

\begin{proposition}\label{prop:Hamming2transOnQ}\cite[Proposition~2.7]{ef2nt}
 Let $\C$ be a $G$-neighbour-transitive code, with minimum distance $\delta\geq 3$, in $\Gamma=H(\N,\Q)=H(n,q)$, and let $i\in \N$. Then $G_i^{\Q_i}$ acts $2$-transitively on $\Q_i$.
\end{proposition}

By an old theorem of Burnside (\cite[Section 154]{burnside1911theory}, or see \cite[Theorem 3.21]{PS2018}) every finite $2$-transitive group is either a group of affine transformations of a finite vector space, or is an almost-simple group. Thus, Proposition~\ref{prop:Hamming2transOnQ} implies that every $(G,2)$-neighbour-transitive code satisfies precisely one of the conditions in Definition~\ref{efaasaadef} below. 

\begin{definition}\label{efaasaadef}
 Let $\C$ be a $G$-neighbour-transitive code in $H(\N,q)$, let $K$ be the kernel of the action of $G$ on $\N$, and let $i\in \N$. Then precisely one of the following holds for $(\C, G)$; i.e.  $\C$ is 
 \begin{enumerate}[(1)]
  \item \emph{$G$-entry-faithful} if $G$ acts faithfully on $\N$, that is, $K=1$;
  \item \emph{$G$-alphabet-almost-simple} if $K\neq 1$, $G$ acts transitively on $\N$, and $G_i^{\Q_i}$ is a $2$-transitive almost-simple group; and
  \item \emph{$G$-alphabet-affine} if $K\neq 1$, $G$ acts transitively on $\N$, and $G_i^{\Q_i}$ is a $2$-transitive affine group.
 \end{enumerate}
 We use descriptors such as $G$-entry-faithful for arbitrary code-group pairs $(\C,G)$ in Hamming graphs. 
\end{definition}

Codes that are $G$-entry-faithful and $(G,2)$-neighbour-transitive were considered by the authors with Gillespie and Giudici in \cite{ef2nt}.

\begin{theorem}\cite[Theorem~1.1]{ef2nt} \label{thm:EntryFaithful}
 Suppose that $\C$ is a code in $H(n,q)$, with $|\C|\geq 2$ and minimum distance $\delta\geq 5$. Then $\C$ is $G$-entry-faithful and $(G,2)$-neighbour-transitive if and only if $\C$ is equivalent to either:
 \begin{enumerate}[$(1)$]
  \item a binary repetition code with $\delta=n$, or,
  \item the even weight subcode of the punctured Hadamard code of length $12$, $G\cong \mg_{11}$, and $\delta=6$.
 \end{enumerate}
\end{theorem}


\subsection{Alphabet-almost-simple codes in Hamming graphs}\label{sec:as}

It turns out that permutation codes (see \cite{bailey2009permutation,blake1979coding} and Definition~\ref{def:permutationCodes}) provide useful building blocks for constructing a wide range of neighbour-transitive codes, many of which are alphabet-almost-simple, as defined in Definition~\ref{efaasaadef}.

\begin{definition}\label{def:permutationCodes}
 Let $\N=\Q=\{1,\ldots,q\}$ and let $T\subseteq \Sym(\Q)$. Then, using the `function' Definition~\ref{defHamming}(2) for the Hamming graph $H(\N,\Q)$, the \emph{permutation code} $\C(T)$ is the code comprising the elements of $T$ considered as functions from $\N$ to $\Q$, that is to say, $t\in T$ corresponds to the codeword $\alpha_t:i\to i^t$ (for $i\in\N$).  In terms of the `$q$-tuple' Definition~\ref{defHamming}(1) for $H(\N,\Q)$, the codeword $\alpha_t$ in $\C(T)$ corresponding to $t\in T$ is $\alpha_t=(1^t,2^t,\ldots,q^t)$, that is to say, $\alpha_t(i)=i^t$. 
\end{definition}

Here is a simple example of a permutation code in $H(4,4)$.

\begin{example}\label{exam:permutationCodes}
 Letting $T=\alt_4$, the permutation code $\C(T)$ consists of the following $4$-tuples:
 \[
  \begin{array}{cccc}
      (1,2,3,4) & (2,1,4,3) & (3,4,1,2) & (4,3,2,1)\\
      (3,1,2,4) & (4,1,3,2) & (4,2,1,3) & (1,4,2,3)\\
      (2,3,1,4) & (2,4,3,1) & (3,2,4,1) & (1,3,4,2).
  \end{array}
 \]
\end{example}

Next we identify some automorphisms of $\C(T)$ among the following elements of $\Aut(H(q,q))=B\rtimes L$, by examining their actions on `permutation' type vertices, that is, vertices with pairwise distinct entries, and hence of the form $\alpha_t$ for some $t\in\Sym(\Q)$.  For an element $g\in\Sym(\Q)$, we write $\sigma_g$ for the corresponding element of the top group $L=\Sym(\N)$, we write $x_g=(g,g,\dots,g)$ for the corresponding `diagonal' element of the base group $B=\Sym(\Q)^q$. The actions of these elements on $H(\N,\Q)$ are given in \eqref{e:hamaut} and \eqref{e:hamaut2}. For any Hamming graph $H(\N,\Q)=H(n,q)$ and any subgroup  $H\leq \Sym(\Q)$, we define the \emph{diagonal subgroup} ${\rm Diag}_n(H)$ of the base group $B=\Sym(\Q)^n$ of $\Aut(H(n,q))$ by 
\begin{equation}\label{e:diag}
    {\rm Diag}_n(H)=\{ x_g = (g,g,\dots, g)\in B\mid g\in H\}.
\end{equation}
Proofs of the following assertions are given in \cite[Lemma 8]{gillespieCharNT}, and are easily derived.

\begin{enumerate}[(1)]
 \item For $g\in \Sym(\Q)$, the element $\sigma_g$ maps $\alpha_t\in\C(T)$ to $\alpha_{tg^{-1}}$. 
 \item For $g\in \Sym(\Q)$, the element $x_g$ maps $\alpha_t\in\C(T)$ to $\alpha_{tg}$.
 \item For $g\in\Sym(\Q)$, the product $x_g\sigma_g$ maps $\alpha_t\in\C(T)$ to $\alpha_{g^{-1}tg}$.
\end{enumerate}
Now suppose that $T$ is a subgroup of $\Sym(\Q)$. Then choosing $g$ to lie in $T$ in parts (1) and (2) we see that $\Aut(\C(T))$ contains ${\rm Diag}_q(T)\times \{\sigma_g\mid g\in T\}\cong T\times T$. Moreover, choosing $g$ in the normaliser $N_{\Sym(\Q)}(T)$ in part (3), we obtain the larger  subgroup ${\rm Diag}_q(T)\rtimes A(T)$ of $\Aut(\C(T))$, where $A(T) :=\{x_g\sigma_g\mid g\in N_{\Sym(\Q)}(T)\}\cong N_{\Sym(\Q)}(T)$. This group is called the holomorph of $T$, \cite[Section 3.3]{PS2018}, and we note that ${\rm Diag}_q(T)\rtimes A(T)\leq {\rm Diag}_q(\Sym(q))\rtimes L$ (recalling that  $L=\Sym(q)$ is  the top group). 

\begin{definition}\label{def:diagxsnt}
A code $\C$ in $H(q,q)$ is \emph{diagonally $s$-neighbour-transitive} if $\C$ is $(G,s)$-neighbour-transitive for some $G\leq \Diag_n(\Sym(\Q))\rtimes L$. If $s=1$ we say simply that $\C$ is diagonally neighbour-transitive.
\end{definition}

In \cite{gillespiediadntc} necessary and sufficient conditions on $T$ were obtained for the code $\C(T)$ to be diagonally neighbour-transitive.

\begin{theorem}\cite[Theorem~2]{gillespiediadntc}\label{thm:2transitivePermcodes}
 Let $\N=\Q=\{1,2,\dots,q\}$, let $T$ be a subgroup of $\Sym(\Q)$,  and let $\C(T)$ be the permutation code as in Definition~\ref{def:permutationCodes}. Then $\C(T)$ is diagonally neighbour-transitive if and and only if $N_{\Sym(\Q)}(T)$ is $2$-transitive. 
\end{theorem}

Examples for which this 2-transitivity condition holds include elementary abelian groups $T$ acting regularly, as well as almost simple $2$-transitive groups. We note that all finite $2$-transitive groups are known explicitly, see for example \cite{cameron1999permgps}, and that this classification depends on the finite simple group classification. The diagonally neighbour-transitive codes $\C(T)$ are building blocks for constructing \emph{frequency permutation arrays} (codes in $H(mq,q)$ where each element of $\Q$ occurs $m$ times as an entry in each codeword; introduced in \cite{freqpermarrays2006}). It was shown in \cite[Theorem~2]{gillespiediadntc} that, for each  diagonally neighbour-transitive   permutation code $\C(T)$ and each positive integer $m$, the repetition code ${\rm Rep}_m(\C(T))$ (see Definiiton~\ref{def:productandrepetition}) is a diagonally neighbour-transitive frequency permutation array in $H(mq,q)$. In particular, when $N_{\Sym(\Q)}(T)$ is an almost  simple $2$-transitive group, all of these diagonally neighbour-transitive codes, $\C(T)$ and ${\rm Rep}_m(\C(T))$, are alphabet-almost-simple as in Definition~\ref{efaasaadef}. 
In fact, these may be regarded as archetypical examples for alphabet-almost-simple codes with minimum distance at least three. In the following result, part (1) was proved first in \cite[Section 7]{gillespieCharNT} and more succinctly in \cite[Proposition 3.3]{aas2nt}; while part (2) was proved in \cite[Theorem 1.1]{gillespieCharNT}.


%




\begin{theorem} \label{partimpdiag}
 Let $\C$ be a code in $H(\N,\Q)=H(n,q)$ which is $G$-alphabet-almost-simple and $G$-neighbour-transitive with $\delta\geq 3$. Then 
 \begin{enumerate}[$(1)$]
     \item there is a $G$-invariant partition $\J$ of $\N$ such that, for $J\in \J$, the projection  $\pi_{J}(\C)$ is equivalent to a diagonally $\chi_{J}(G)$-neighbour-transitive code with minimum distance $\delta(\pi_{J}(\C))\geq 2$;
     \item $\C$ has a sub-code $\SS$ which is a neighbour-transitive frequency permutation array, and $\C$ is a disjoint union of $G$-images of $\SS$.
 \end{enumerate}
\end{theorem}

We make some comments about the links between the two parts of Theorem~\ref{partimpdiag}, and the structure of the sub-code $\SS$ in part (2).

\begin{remark}\label{rem:aas}
(a)  It follows from \cite[Theorems 1 and 3]{gillespiediadntc} that the  projected codes $\pi_J(\C)$ in part (1) of Theorem~\ref{partimpdiag}  are either frequency permutation arrays or repetition codes.

\smallskip
(b) Some explicit information about the possibilities for the neighbour-transitive sub-code $\SS$  in part (2) of Theorem~\ref{partimpdiag} is provided by \cite[Theorem 1.1]{gillespieCharNT} as follows: $\SS$ is equivalent to a code of the form $\Prod_\ell(\Rep_k(\C'))$ (with $\ell, k\geq1$), where $\Prod_\ell$ and $\Rep_k$ are the product and repetition constructions defined in Definition~\ref{def:productandrepetition}. Moreover, the small code $\C'$ itself has one of three specific forms described in \cite[(7,4), (7.5), (7.6)]{gillespieCharNT}, namely the trivial code $\Q$ of length $1$, or a permutation code $\C(T)$ as in Definition~\ref{def:permutationCodes}, or a twisted version $\C(T,T^\sigma)$ of a permutation code defined and studied in \cite{akbari2018permutationcodes, gillespie2015twistedPermCodes}. 

\smallskip
(c) The code $\C$ in Theorem~\ref{partimpdiag} may have minimum distance strictly larger than that of the projected codes $\pi_J(\C)$ in part (1). An insightful example was given in  \cite[Example 9.1]{gillespieCharNT}, where $\SS=\Prod_\ell(\C(A_q))$, and $\C$ is a disjoint union of two $G$-translates of $\SS$. Each of the codes $\C$ and $\SS$ has minimum distance $3$, while $\pi_J(\C)$ is the permutation code $\C(S_q)$ with minimum distance $2$.

\end{remark}

We give some explicit constructions of alphabet-almost-simple neighbour-transitive codes, the first built from permutation codes.

\begin{example}\label{exam:productsPermCodes}
 Let $\N=\Q=\{1,\ldots,q\}$, and let $T\leq \s_q$ be such that the permutation code $\C=\C(T)$  in $H(q,q)$ is diagonally neighbour-transitive (see Theorem~\ref{thm:2transitivePermcodes}). Then, letting $\M=\N\times\{1,\ldots,k\}$, as in Definition~\ref{def:productandrepetition}, the product code $\Prod_k(\C)$ and the repetition code $\Rep_k(\C)$ are neighbour-transitive codes in $H(\M,\Q)$. Moreover, for each of these codes, and for each $i=1,\ldots,k$, the projection with respect to $\N\times\{i\}$ is equal to $\C(T)$. The following codes are additional examples of codes that again project to $\C(T)$  with respect to $\N\times\{i\}$, and properly lie between $\Rep_k(\C)$ and $\Prod_k(\C)$. Recall from Definition~\ref{def:productandrepetition} that if $g_1,\ldots,g_k$ are functions $\N\rightarrow\Q$ then $\beta_{g_1,\ldots,g_k}$ is a vertex of $H(\M,\Q)$, and note that if $g\in T$ then $g$ is a function $\Q\rightarrow\Q$.
 \begin{enumerate}[(1)]
  \item Let $T=\s_q$ and consider the code consisting of all $\beta_{t_1,\ldots,t_k}$ such that $t_1,\ldots,t_k\in T$ and $t_1t_2\cdots t_k\in\alt_n$. (If $k$ is even then this is one of the codes corresponding to the `elusive' code $\C(q,k)$ in \cite[Lemma 3.9]{elusive}, from which it follows that $\Prod_k(\C)$ is neighbour-transitive.) 
  \item Let $1\ne H\lhd T$ and let
  \[
   \Prod(T,k,H)=\{\beta_{h_1t,\ldots,h_kt}\mid h_1,\ldots,h_k\in H,\, t\in T\}.
  \]
  We state below that $\Prod(T,k,H)$ is neighbour-transitive as long as $N_{\s_q}(H)$ is $2$-transitive. Note that several explicit examples are given, and precise requirements on $H$ and $T$ related to the neighbour-transitivity of $\Prod(T,k,H)$ are discussed further, in \cite{hawtin2023newpermarrays}. 
 \end{enumerate}
\end{example}



\begin{proposition}\cite{hawtin2023newpermarrays}
 Let $k$ be a positive integer, let $T\leq \s_q$ and let $1\ne H\lhd T$ such that $N_{\s_q}(H)$ is $2$-transitive. Then the code $\Prod(T,k,H)$ in Example~\ref{exam:productsPermCodes}(2) is neighbour-transitive.
\end{proposition}

\begin{problem}\label{prob:freqPermArrays}
 Investigate further the codes $\Prod(T,k,H)$ in Example~\ref{exam:productsPermCodes}.
\end{problem}

The second construction is a general version of the twisted permutation codes $\C(T,T^\sigma)$ from Remark~\ref{rem:aas}(b). Not all of them will arise as the building block $\C'$ for alphabet-almost-simple codes. More details of this construction are available in \cite{akbari2018permutationcodes,gillespie2015twistedPermCodes}. 

\begin{definition}
 Let $T$ be a group with (not necessarily distinct) permutation representations $\rho_1,\ldots,\rho_k:G\rightarrow\s_q$, let $\Q=\{1,\ldots,q\}$ and let $\N=\Q\times\{1,\ldots,k\}$. We define the \emph{twisted permutation code}  $\C(T;\rho_1,\ldots,\rho_k)$ in $H(kq,q)$ as  the code consisting of the functions $\N\rightarrow\Q$ given by $(i,j)\mapsto i^{\rho_j(t)}$, for each $t\in T$.
\end{definition}

One benefit of a twisted permutation code $\C(T;\rho_1,\ldots,\rho_k)$ in $H(kq,q)$ over the usual repetition code $\Rep_k(\C(T))$ in the same graph $H(kq,q)$ is that the minimum distance of the twisted version may be greater than for the untwisted code $\Rep_k(\C(T))$. Some examples are given in Table~\ref{tab:twistedPermCodes}. While it is only the examples in Lines 1--3 which arise in connection with alphabet-almost-simple codes, examples in the other lines, especially Lines 5--6 demonstrate that the differences between the minimum distances of the twisted and untwisted codes can be unbounded. The groups $G_t$ in line 6 are defined and studied in \cite[Section 3.1]{akbari2018permutationcodes}; $G_t$ is a certain subgroup of $\AGL_t(p)$, it contains the translation group and has order $p^{t+1}$.

Among this diverse family of alphabet-almost-simple, neighbour-transitive codes, very few are $2$-neighbour-transitive. In fact we have the following classification of such codes.

\begin{table}
\begin{center}
\begin{tabular}{c|ccccc}
 \hline
 Line & $T$ & $q$ & $k$ & $\delta_{\rm tw}$ & $\delta_{\rm rep}$\\
 \hline
 1 & $\s_6$ & $6$ & $2$ & $8$ & $4$ \\
 2 & $\alt_6$ & $6$ & $2$ & $8$ & $6$ \\
 3 & $\ASL_3(2)$ & $8$ & $2$ & $12$ & $8$ \\
 4 & $\s_6$ & $60$ & $4$ & $\leq 224$ & $176$--$192$ \\
 5 & $\Sp_4(2^n)$ & $2^{3n}+2^{2n}+2^n+1$ & $2$ & $2^{3n+1}+2^{2n}$ & $2^{3n+1}$ \\
 6 & $G_t$ & $p^t$ & $p$ & $p^{t+1}-p$ & $p^{t+1}-p^2$ \\
 \hline
\end{tabular}
\caption{Groups $T$ giving twisted permutation codes in $H(kq,q)$ with minimum distance $\delta_{\rm tw}$ strictly greater than the minimum distance $\delta_{\rm rep}$ of the $k$-fold repetition of $\C(T)$. Lines 1--4 can be found in \cite{gillespie2015twistedPermCodes} while lines 5 and 6 are in \cite{akbari2018permutationcodes}.}
\label{tab:twistedPermCodes}
\end{center}
\end{table}

\begin{theorem}\cite[Theorem~1.1]{aas2nt} \label{thm:2NTAlphabetAlmostSimple}
 Suppose that $\C$ is an alphabet-almost-simple, $2$-neighbour-transitive code in $H(n,q)$ with minimum distance $\delta \geq 3$. Then $n=\delta=3$, $q\geq5$, and $\C$ is equivalent to the repetition code $\Rep_3(q)$.
\end{theorem}

\subsection{Alphabet-affine codes in Hamming graphs}\label{sec:aaff}

In this section we consider codes that are alphabet-affine (Definition~\ref{efaasaadef}), that is, codes $H(n,q)$ having an automorphism group not acting faithfully on entries, and giving rise to a $2$-transitive affine group in the action induced on the alphabet. A natural family of examples of such codes are the cyclic codes: a code $\C$ in $H(\N,\F_q)$
 \begin{equation}\label{def:cyclic}
\text{
\begin{minipage}{14cm}
 is called \emph{cyclic} if $\C$ is linear and there exists an $n$-cycle $\sigma\in L\cap\Aut(\C)$. 
\end{minipage}
}
\end{equation}
In particular each cyclic code $\C$ is alphabet-affine since $\Aut(\C)$ contains the subgroup of translations by elements of $\C$ (acting transitively on $\C$). Moreover, a cyclic code is neighbour-transitive since $\Aut(\C)$ also contains both the subgroup of scalars and an $n$-cycle from the top group $L$.
On the other hand, if $x=h\sigma$ with $h\in B$ and $\sigma\in L$, then $x\in\Aut(\C)$ does not necessarily imply that $\sigma\in\Aut(\C)$. That is to say, a linear code $\C$ may not be cyclic even if the induced group $\Aut(\C)^\N$ on entries contains an $n$-cycle. We exhibit a small non-cyclic linear code with this property in Example~\ref{ex:noncyclic}.  In fact, as we see below, the class of codes that are alphabet-affine and neighbour-transitive is strictly larger than the class of cyclic codes.

\begin{proposition}\label{prop:cyc1}
 Let $\C$ be a linear code with minimum distance $\delta\geq 3$ in $H(\N,\F_q)$. Then $\C$ is neighbour-transitive if and only if $\Aut(\C)^{\N}$ is transitive. In particular, if $\C$ is a cyclic code then $\C$ is neighbour-transitive.
\end{proposition}

\begin{proof}
 If $\C$ is neighbour-transitive then, by Proposition~\ref{prop:HammingiHomogeneous}, $\Aut(\C)^\N$ is transitive. Suppose now that $\Aut(\C)^\N$ is transitive. Since $\C$ is linear, $\Aut(\C)$ contains $T_{\C}$, which acts transitively on $\C$. Moreover, this implies that $\Aut(\C)=T_{\C}.\Aut(\C)_{{\mathbf 0}}$. In particular, $\Aut(\C)_{{\mathbf 0}}^{\N}$ is transitive. For each $i\in \N$, let $e_i:\N\rightarrow \F_q$ such that $e_i:i\to1$ and $e_i:k\to 0$ if $k\neq i$. Then $\Gamma_1({\mathbf 0})=\{ae_i\mid a\in\F_q^\times\}$. Let $i,j\in \N$ and $a_i,a_j\in\F_q^\times$ so that $a_ie_i,a_je_j\in \Gamma_1({\mathbf 0})$. Then, since $\Aut(\C)_{{\mathbf 0}}$ induces a transitive action on $\N$, there exists an $x\in\Aut(\C)_{{\mathbf 0}}$ such that $(a_ie_i)^x=be_j$, for some $b\in\F_q^\times$. Using the linearity of $\C$ again, scalar multiplication by $a_j b^{-1}$ gives an element of $\Aut(\C)_{{\mathbf 0}}$, and hence we can map $a_ie_i$ to $a_je_j$. It then follows from Lemma~\ref{lem:largeDeltaDistTrans} that $\C$ is neighbour-transitive. The last sentence holds since $\C$ being cyclic implies that $\Aut(\C)^\N$ is transitive.
\end{proof}

\begin{example}\label{ex:noncyclic}
 Consider the linear code $\C$ in $H(4,\F_3)$ consisting of the codeword $(0,0,0,0)$ as well as the following $8$ non-zero codewords:
 \[
  \begin{array}{cc}
      (0,1,1,1) & (0,2,2,2) \\
      (1,0,1,2) & (2,0,2,1) \\
      (2,1,0,2) & (1,2,0,1) \\
      (2,2,1,0) & (1,1,2,0).
  \end{array}
 \]
 The automorphism $x:(a,b,c,d)\mapsto (d,a,b,2c)$ of $\C$ fixes $(0,0,0,0)$ and cycles through the $8$ non-zero codewords of $\C$; $x$ has order $8$ and projects to the $4$-cycle $(1\,2\,3\,4)$ in the top group $L\cong\s_4$. However $\C$ is not cyclic, for if $\sigma$ is a $4$-cycle in $L$ then, since the only codeword consisting of one $0$ and three $1$s is $(0,1,1,1)$, the image $(0,1,1,1)^\sigma$ cannot lie in $\C$.  On the other hand, $\C$ has minimum distance $\delta=3$ and satisfies all the hypotheses of Proposition~\ref{prop:cyc1}, so it is neighbour-transitive. Moreover, $\C$ is not equivalent to any cyclic code (one way to see this is to observe that the subspace spanned by $\{\alpha^{\sigma^k}\mid k=0,1,2,3\}$ contains a weight $2$ vector, for each of $\alpha=(0,1,1,1),(0,1,1,2),(0,1,2,1)$ and where $\sigma=(1\,2\,3\,4)$). Note also that $\C$ is the projective Reed--Muller code $\PRM_3(1,2)$, which we meet later on in Definition~\ref{def:projReedMuller}.  
\end{example}

The next result suggests investigating the submodule structure of certain modules, which we consider further in subsequent sections. For a group $G$ and a prime $p$, $O_p(G)$ denotes the largest normal $p$-subgroup of $G$. 


\begin{proposition}\cite[Proposition~3.5]{minimal2nt}\label{orbitof0is2ntmodule}
  Let $\C$ be a  code in the Hamming graph $H(n,q)$, with $q=p^d$ for a prime $p$, such that $\C$ is $G$-alphabet-affine and $(G,2)$-neighbour-transitive, with $\delta\geq 5$, and suppose that $\b 0\in C$. Then $\C$ contains a subcode $\S$ such that $\S$ is the code formed by the orbit of $\b 0$ under $O_p(K)$, where $K=G\cap B$. Moreover, it follows that: 
 \begin{enumerate}[$(1)$]
  \item $\S$ is a block of imprimitivity for the action of $G$ on $\C$, and $G_{\S}=O_p(K)\rtimes G_{\b 0}$,
  \item $\S$ is $G_{\S}$-alphabet-affine and $(G_{\S},2)$-neighbour-transitive with minimum distance $\delta_{\S}\geq \delta$,
  \item $\S$ is an $\F_p G_{\b 0}$-module, and if $\S\neq \Rep_n(2)$ then $q^2$ divides $|\S|$.
 \end{enumerate}
\end{proposition}




The following result gives further, fairly strong, restrictions on the structure of the automorphism group of a $2$-neighbour-transitive code. 

\begin{lemma}\cite{hawtin2023alphabetaffine}\label{lem:k0solubleandsemireg}
 Let $\C$ be a  code in $H(n,q)=H(\N,\Q)$ that is  $(G,2)$-neighbour-transitive with minimum distance $\delta\geq 5$. 
 \begin{enumerate}[$(1)$]
     \item If in addition $\C$ is $G$-alphabet-affine, then  $\Gpi\leq \GaL_1(q)$.
     \item Alternatively, if  ${\b 0}\in\C$ and  $K=B\cap G$, then $K_{\b 0}\cong \Diag_n(H)$, where $H$ acts semi-regularly on ${\Q}_i^\times$ for all $i\in \N$. 
 \end{enumerate}
\end{lemma}

\subsection{Codes in binary Hamming graphs}\label{sec:bin}

When we restrict the alphabet $\Q$ to be the field $\F_2$ of order $2$ then we obtain very tight descriptions of the possibilities for $(G,s)$-neighbour-transitive codes with minimal distance not too small. Recall from Proposition~\ref{prop:HammingiHomogeneous} that for such codes, the stabiliser in $G$ of a codeword induces an $s$-homogeneous action on entries. Also, for a code $\C\leq \F_2^n$, $T_\C$ denotes the set of all translations by elements of $\C$; if $T_\C$ is a subgroup of $\Aut(\C)$ then $\C$ is additively closed. 
The next result gives a  nearly complete description of the binary $2$-neighbour-transitive codes with $\delta$ at least $5$.

\begin{theorem}\cite[Theorem~1.2]{minimal2nt}\label{binaryx2ntchar}
 Let ${\mathcal C}$ be a code in $H(n,2)$ with minimum distance $\delta\geq 5$. Then $\C$ is $2$-neighbour-transitive if and only if one of the following holds:
 \begin{enumerate}[$(1)$]
  \item ${\mathcal C}$ is the binary repetition code $\Rep_n(2)$ with $\delta=n$.
  \item ${\mathcal C}$ is one of the following codes (see \cite[Definition~4.1]{ef2nt}): 
  \begin{enumerate}[$(a)$]
   \item the Hadamard code with $n=12$ and $\delta=6$;
   \item the punctured Hadamard code with $n=11$ and $\delta=5$;
   \item the even weight subcode of the punctured Hadamard code with $n=11$ and $\delta=6$.
  \end{enumerate}
  \item There exists a linear subcode $\S \subseteq {\mathcal C}$ with dimension $k$ and minimum distance $\delta$, and a subgroup $G_{\b 0}\leq \Aut(\C)_{\b 0}$, where $\S$, $G_{\b 0}$, $n, \delta$ and $k$ are as in Table~$\ref{binarytable}$, such that
  \begin{enumerate}[$(a)$]
   \item $\S$ is $(G,2)$-neighbour-transitive, where $G=T_{\S}\rtimes G_{\b 0}\leq \Aut(\C)$, and,
   \item $\C$ is the union of a set $\Delta$ of cosets of $\S$, and $\Aut(\C)$ acts transitively on $\Delta$.
  \end{enumerate}
 \end{enumerate}
\end{theorem}



\begin{remark}\label{psuremark}
 Table~\ref{binarytable} gives the possibilities for the linear subcode $\S$ of $\C$ in Theorem~\ref{binaryx2ntchar}(3)(a). For each line of Table~\ref{binarytable}, the code $\C$ in Theorem~\ref{binaryx2ntchar}(3)(b) is identified in the relevant part of the proof of \cite[Theorem~4.5]{minimal2nt}.  We note that the minimum distance $\delta$ of $\C$ satisfies $5\leq\delta< n$ in all lines except possibly line 9. In both of the lines 8 and 9 of Table~\ref{binarytable} we have $\soc(G_{\b 0})=\PSU_3(r)$, and it follows from the  proof of \cite[Theorem~4.5]{minimal2nt} that $\S$ is self-orthogonal (that is, $\S\subseteq \S^\perp$) if $r\equiv 3\pmod{4}$ but not if $r\equiv 1\pmod{4}$. Self-orthogonality ensures that only the example in line 8 of a minimal $(G,2)$-neighbour-transitive code arises when $r\equiv 3\pmod{4}$, while for $r\equiv 1\pmod{4}$ we have two different examples, one for each of lines 8 and 9 of Table~\ref{binarytable}. It was proved, see \cite[Remark~1.3]{minimal2nt}, that the codes in line 9 have minimum distance $\delta\geq 4$, and it is an open problem to determine precisely which values of $r\equiv 1\pmod{4}$ correspond to a code with $\delta\geq 5$.  
\end{remark}

\begin{table}
 \begin{center}
 \begin{tabular}{c|ccccc}
  \hline
  Line & $\soc(G_{\mathbf 0})$ & $n$ & $\delta$ & $k$ & Conditions\\
  \hline
  1 & $\Z_p^d$ & $r=p^d$ & $\geq (r-1)^{1/2}+1$ & $\frac{r-1}{2}$ & $23\leq r\equiv 7 \pmod{8}$ \\
   & &  &   &  & $2$-hom. not $2$-trans.\\
  
  2 & $\Z_2^t$ & $2^t$ & $2^{t-1}$ & $t+1$ & $t\geq 4$, $2$-trans. \\
  
  3 & $\PSL_t(2^a)$ & $\frac{2^{at}-1}{2^a-1}$ & $\geq\frac{2^{a(t-1)}-1}{2^a-1}+1$ & $t^a$ & $t\geq 3$, $(a,t)\neq (1,3)$ \\
  
  4 & $\alt_7$ & $15$ & $8$ & $4$ & - \\
  
  5 & $\PSL_2(r)$ & $r+1$ & $\geq r^{1/2}+1$ & $\frac{r+1}{2}$ & $23\leq r\equiv\pm 1\pmod 8$ \\
   & &  &  &  & not $3$-trans.\\
  
  6 & $\Sp_{2t}(2)$ & $2^{2t-1}-2^{t-1}$ & $2^{2t-2}-2^{t-1}$ & $2t+1$ & $t\geq 3$ \\
  7 & $\Sp_{2t}(2)$ & $2^{2t-1}+2^{t-1}$ & $2^{2t-2}$ & $2t+1$ & $t\geq 3$ \\
  
  8 & $\PSU_3(r)$ & $r^3+1$ & $\geq r^2+1$ & $r^2-r+1$ & $r$ is odd \\
  9 & $\PSU_3(r)$ & $r^3+1$ & $\geq 4$ & $r^3-r^2+r$ & $r\equiv 1\pmod{4}$ \\
  
  10 & $\Ree(r)$ & $r^3+1$ & $\geq r^2+1$ & $r^2-r+1$ & $r\geq 3$ \\
  
  11 & $\mg_{22}$ & $22$ & $8$ & $10$ & - \\
  12 & $\mg_{23}$ & $23$ & $8$ & $11$ & - \\
  13 & $\mg_{24}$ & $24$ & $8$ & $12$ & - \\
  
  14 & $\HS$ & $176$ & $\geq 50$ & $21$ & - \\
  
  15 & $\Co_3$ & $276$ & $100$ & $23$ & - \\
  \hline
 \end{tabular}
 \caption{Parameters for the $(G,2)$-neighbour-transitive code $\S$ in Theorem~\ref{binaryx2ntchar}(3).  See  Remark~\ref{psuremark} for more information.}
 \label{binarytable}
 \end{center}
\end{table}

In Theorem~\ref{binaryx2ntchar}(3) the linear subcode $\S$ (see Proposition~\ref{orbitof0is2ntmodule}) is a submodule of the permutation module over $\F_2$ of a $2$-homogeneous permutation group. The next lemma explores the properties of such subcodes.


\begin{lemma}\cite[Lemma~4.3]{minimal2nt}\label{lem:moduleiscode}
 Let $H$ act $2$-homogeneously on a set $\N$ of size $n\geq 5$, let $V\cong\F_2^n$ be the permutation module for the action of $H$ on $\N$, and let $\C$ be a proper $\F_2 H$-submodule of $V$. 
 Then $\C$ is a code in $H(\N,\F_2)$ with minimum distance $\delta$ and the group $G :=T_{\C}\rtimes H\leq \Aut(\C)$, where precisely one of the following holds:
 \begin{enumerate}[$(1)$]
  \item $\C, \delta, G$ satisfy one of the lines of Table~\ref{moduletable};
  \item $\delta=3$, $\C$ is a perfect code in $H(\N,\F_2)$, and $\C$ is $G$-neighbour-transitive; or
  \item $4\leq \delta< n$ and $\C$ is $(G,2)$-neighbour-transitive.
 \end{enumerate}
\end{lemma}

\begin{table}
 \begin{center}
 \begin{tabular}{ccc}
  \hline
  $\C$ & $\delta$ & Properties\\
  \hline
   $\Rep_n(2)$ & $n$ & $(G,2)$-neighbour-transitive \\
   $\Rep_n(2)^\perp$ & $2$ & $G$-completely-transitive\\
  \hline
 \end{tabular}
 \caption{Codes arising in Lemma~\ref{lem:moduleiscode}(1) and their properties.}
 \label{moduletable}
 \end{center}
\end{table}

Every perfect code with minimum distance $3$ in $H(n,2)$ has the same parameters as a Hamming code, see \cite{VL75}. Thus, in a sense, Lemma~\ref{lem:moduleiscode} part (2) is well understood, as is part (1). In addition, those codes in Lemma~\ref{lem:moduleiscode} part (3) having minimum distance at least $5$ are described in Theorem~\ref{binaryx2ntchar}. Thus, the codes in Lemma~\ref{lem:moduleiscode} for which  least is known are those in part (3) with $\delta=4$. These are the subject of Problem~\ref{prob:binaryDeltaEquals4}.

\begin{problem}\label{prob:binaryDeltaEquals4}
 Determine all codes with minimum distance $\delta=4$ corresponding to submodules of the permutation module over $\F_2$ of a $2$-homogeneous permutation group.
\end{problem}

Note that a completely transitive code $\C$ in $H(n,2)=H(\N,2)$ with minimum distance $\delta\geq 5$ has covering radius $\rho\geq 2$, and hence is $2$-neighbour-transitive so that Theorem~\ref{binaryx2ntchar} may be applied. In particular, if $\C\neq\Rep_n(2)$, then the induced subgroup $S\cong\soc(\Aut(\C)_{\b 0}^\N)$ is as in Theorem~\ref{binaryx2ntchar}(2) or (3), so $S$ is either a small Mathieu group $\mg_{11}, \mg_{12}$, or $S$ is one of the groups in the column `$\soc(G_{\b 0})$' of Table~\ref{binarytable}. This observation was exploited in \cite{bailey2022classification} to obtain a partial classification, stated in Theorem~\ref{binaryCTclass}, of binary completely transitive codes with minimum distance at least $5$.

\begin{theorem}\cite[Theorem~1.3]{bailey2022classification}\label{binaryCTclass}
 Let $\C$ be a non-trivial completely transitive code in $H(n,2)$ with minimum distance $\delta\geq 5$, and suppose that the socle $S$ of the group induced by $\Aut(\C)_{\b 0}$ on $\N$ is as in one of the lines of Table~$\ref{tab:binCTgroups}$. Then $\C$ is equivalent to one of the codes in Table~$\ref{binaryCTtable}$. Moreover, each code in Table~\ref{binaryCTtable} is completely transitive.
\end{theorem}

\begin{table}
 \begin{center}
 \begin{tabular}{ccc}
  \hline
  $S$ & $n$ & Conditions\\
  \hline
  $\PSL_3(4)$ & $21$ & - \\
  
  $\alt_7$ & $15$ & - \\
  
  $\PSL_2(r)$ & $r+1$ & $23\leq r\equiv\pm 1\pmod 8$ \\
  
  
  $\PSU_3(r)$ & $r^3+1$ & $r$ is odd \\
  
  $\Ree(r)$ & $r^3+1$ & $r\geq 3$ \\
  
  $\mg_{m}$ & $11,12,22,23,24$ & - \\
  
  $\HS$ & $176$ & - \\
  
  $\Co_3$ & $276$ & - \\
  \hline
 \end{tabular}
 \caption{Some groups $S$ for which there exists a non-trivial binary $2$-neighbour-transitive code $\C$ in $H(n,2)$ such that $\C\neq\Rep_n(2)$ and $S\cong\soc(\Aut(\C)_{\b 0}^\N)$.}
 \label{tab:binCTgroups}
 \end{center}
\end{table}


\begin{table}
 \begin{center}
 \begin{tabular}{c|ccccc}
  \hline
  Line & $\C$ & $\Aut(\C)$ & Parameters \\
  \hline  
  1 & $\Had$ & $2\mg_{12}$ & $(12,24,6;3)$ \\
  
  2 & $\PHad$ & $2\rtimes \mg_{11}$ & $(11,24,5;3)$ \\
  
  3 & $\NR$ & $2^5\rtimes\alt_8$ & $(15,256,5;3)$ \\
  
  4 & $\langle\L,\Delta_1\rangle\cup\langle\L,\Delta_2\rangle$ & $T_\L\rtimes \PGaL_3(4)$ & $(21,2^{10}\cdot 3,5;6)$ \\
  
  5 & $\P^\perp$ & $T_C\rtimes \PGaL_3(4)$ & $[21,12,5;3]$ \\
  
  6 & $\langle\L,\Delta_1\rangle$ & $T_C\rtimes \PSiL_3(4)$ & $[21,11,5;6]$ \\
  
  7 & $\L$ & $T_C\rtimes \PGaL_3(4)$ & $[21,10,5;6]$ \\
  
  8 & $\G_{24}$ & $T_C\rtimes \mg_{24}$ & $[24,12,8;4]$ \\
  
  9 & $\G_{23}$ & $T_C\rtimes \mg_{23}$ & $[23,12,7;3]$ \\
  
  10 & $\G_{23}^\perp$ & $T_C\rtimes \mg_{23}$ & $[23,11,8;7]$ \\
  
  11 & $\G_{22}$ & $T_C\rtimes (\mg_{22}:2)$ & $[22,12,6;3]$ \\
  
  12 & $\EG_{22}$ & $T_C\rtimes (\mg_{22}:2)$ & $[22,11,6;7]$ \\
  
  13 & $\SG_{22}$ & $T_C\rtimes \mg_{22}$ & $[22,11,7;6]$ \\
  
  \hline
 \end{tabular}
 \caption{Non-trivial binary completely transitive codes $\C$ with minimum distance $\delta\geq 5$ and where $\Aut(\C)_{\b 0}^{\N}$ has as socle one of the groups $S$ in Table~\ref{tab:binCTgroups}. See \cite[Section~3]{bailey2022classification} for the definitions of these codes. The codes in Lines 1--4 are non-linear with parameters $(n,|\C|,\delta;\rho)$, while the remaining codes are linear with parameters $[n,k,\delta;\rho]$, where $\rho$ is the covering radius of $\C$ and $k$ is the dimension.}
 \label{binaryCTtable}
 \end{center}
\end{table}


Note that, of the groups $S$  occurring in Theorem~\ref{binaryx2ntchar}(2) or (3), the only ones not  considered in Theorem~\ref{binaryCTclass} are the following groups `$\soc(G_{\b 0})$' of Table~\ref{binarytable}: i) $\PSL_t(2^a)$ unless $(t,a)=(3,2)$, ii) $\Sp_{2t}(2)$ for $t\geq 3$, iii) $\Z_2^t$ for $t\geq 3$, or iv) $\Z_p^d$ where $r=p^d\equiv 7\pmod{8}$.  Thus the following problem is still open.


\begin{problem}\label{prob:binaryCT}
 Classify the binary completely transitive codes $\C$  in $H(n,2)$ with minimum distance at least $5$, for which $\soc(\Aut(\C)_{\b 0}^\N)$ is one of i) $\PSL_t(2^a)$ with $(t,a)\neq(3,2)$, ii) $\Sp_{2t}(2)$ for $t\geq 3$, iii) $\Z_2^t$ for $t\geq 3$, or iv) $\Z_p^d$ where $23\leq n=p^d\equiv 7\pmod{8}$.
\end{problem}

\subsection{A non-existence result: proof of Theorem~\ref{thm:CTham}}\label{sec:HammingNonexistence}

We now have enough information about $G$-alphabet-affine and $(G,s)$-neighbour-transitive codes to prove Theorem~\ref{thm:CTham} which was stated in Section~\ref{sec:ct}.

\begin{proof}[Proof of Theorem~\ref{thm:CTham}]

 If $n\geq 9$ and $\C=\Rep_n(2)$ in $H(n,2)$ then, by Theorem~\ref{thm:EntryFaithful}, $\C$ is completely transitive with minimum distance $\delta=n$, and hence error-correction capacity $e=\lfloor \frac{n-1}{2}\rfloor$. Since $\Gamma_i({\b 0})$ consists precisely of the weight $i$ vertices of $H(n,2)$ and $\Gamma_i({\b 0})$ is contained in $\C_i$ if and only if $i\leq n/2$, the largest value for $s$ for which $\C$ is $s$-neighbour-transitive is $s=\lceil \frac{n-1}{2}\rceil$. Thus $\min\{e,s\}\geq4$ and all the conditions of Theorem~\ref{thm:CTham} hold.
 
 Suppose from now on that $\C$ is an $s$-neighbour-transitive code in $H(n,q)$ with error capacity $e$ such that $\min\{e,s\}\geq4$, and that $\C\neq\Rep_n(2)$. Let $G\leq \Aut(\C)$ such that $\C$ is $(G,s)$-neighbour-transitive. Since $e\geq 4$ we have $\delta\geq 9$, which also implies that $n\geq 9$. It now follows from Theorem~\ref{thm:EntryFaithful} that, if $\C$ is $G$-entry-faithful then $\C$ is equivalent to $\Rep_n(2)$. Let us assume now that  $\C$ is not $G$-entry-faithful, and that $\C$ is not equivalent to $\Rep_n(2)$.  Since $\delta\geq 9$, it follows from Theorem~\ref{thm:2NTAlphabetAlmostSimple} that $\C$ is not $G$-alphabet-almost-simple and thus, by Proposition~\ref{prop:Hamming2transOnQ} and Definition~\ref{efaasaadef},  $\C$ is $G$-alphabet-affine.  
 
 Replacing $\C$ by an equivalent code we may assume that $\b 0\in\C$. By Proposition~\ref{prop:HammingiHomogeneous}, $G_{\b 0}$ acts $4$-homogeneously on $\N$. Hence, by \cite[Table~7.4]{cameron1999permgps} and \cite{kantor1972k}, $G_{\b 0}^\N$ is one of the groups in Table~\ref{tab:4homgroups} of degree $n$. 
 Also, by Proposition~\ref{orbitof0is2ntmodule}, there exists a subcode $\S$ of $\C$ such that $\S$ is an $\F_p G_{\b 0}$-module. If $q=2$, then Theorem~\ref{binaryx2ntchar} eliminates each possibility for $G_{\b 0}^\N$ (recalling where necessary that $\delta\geq 9$). Thus $q\geq 3$.

 Let $I\subseteq \N$ with $|I|=4$. Now $\C$ is $(G,4)$-neighbour-transitive (as $s\geq4$), so $G_{\b 0}$ is transitive on the set of all weight $4$ vertices of $H(n,q)$. Also, as $G_{\b 0}$ is $4$-homogeneous on $\N$, the setwise stabiliser $G_{{\b 0},I}$ acts transitively on the set of weight $4$ vertices having support $I$. Hence $(q-1)^4$ divides the order of $G_{\b 0,I}^{H(I,q)}$. Now by Lemma~\ref{lem:k0solubleandsemireg}(1) the induced group $G_{\b 0,I}^{H(I,q)}$ is a subgroup of $\GaL_1(q)\wr\s_4$, and in particular is soluble. Also, by Lemma~\ref{lem:k0solubleandsemireg}(2),  the group $K_0:=B\cap G_{\b 0}$ has order dividing $q-1$ and $K_{\b 0}\cong K_{\b 0}^{H(I,q)}$. By the definition of $K_0$ we have $G_{\b 0, I}^\N\cong G_{\b 0, I}/K_{\b 0}$, and hence  $G_{\b 0,I}^{H(I,q)}/K_{\b 0}^{H(I,q)}$ is a soluble quotient of $G_{\b 0,I}^\N$ with order divisible by $(q-1)^3$. 

 Recall that $n\geq9$, and let ${\rm Sol}$ be the order of the largest soluble normal subgroup of $G_{\b 0,I}^\N$, For example, if $G_{\b 0}^\N = {\rm S}_n$ then $G_{\b 0,I}^\N = \s_4\times \s_{n-4}$ so ${\rm Sol}=2^4\cdot 3$.  Since $(q-1)^3$ divides ${\rm Sol}$ it follows that the possibilities for $q$ are as in Table~\ref{tab:4homgroups}. In particular the last three lines of Table~\ref{tab:4homgroups} are ruled out as there are no possibilities for $q$. If $q$ is $3$ or $5$ then $K_0$ consists of scalars and the subcode $\S$ of $\C$ mentioned above is an $\F_q G_{\b 0}^\N$-module, of dimension $k$, say. In particular $\S$ is a linear $[n,k,\delta']$ code with $\delta'\geq \delta\geq 9$, so by the Singleton bound \cite[Theorem~11, Chapter~1]{macwilliams1978theory}, $k\leq n-\delta +1\leq n-8$. This gives an immediate contradiction for the first four lines of Table~\ref{tab:4homgroups}: for $\alt_n$ and $\s_n$ by \cite[Proposition~5.3.7]{Kleidman-Liebeck}, for $\mg_{11}$ and $\mg_{12}$ by \cite[Proposition~5.3.8]{Kleidman-Liebeck}. The remaining cases are $G_{\b 0}^\N=\mg_{23}$ or $\mg_{24}$. A similar argument to that in the previous paragraph, for a subset $J\subseteq \N$ with $|J|=2$ shows that $q-1$ divides the order of a soluble quotient of $G_{\b 0,J}^\N$. However, for  $G_{\b 0}^\N=\mg_{23}$ or $\mg_{24}$, the largest soluble quotient of $G_{\b 0,J}^\N$ has order 2. Thus $q=3$, and again the subcode $\S$ of $\C$ is an $\F_q G_{\b 0}^\N$-module, of dimension $k$, say. This implies, by \cite{atlas1995}, that $k\geq 22$, whereas the Singleton bound requires $k\leq n-\delta+1\leq 24-8=16$. 
 \end{proof}
 
\begin{table}
 \begin{center}
 \begin{tabular}{cccc}
  \hline
  $G_{\b 0}^\N$ & $n$ & {\rm Sol}& $q\geq3$\\
  \hline
  $\alt_n$ & $n$ & $2^3\cdot 3$ & 3 \\
  
  ${\rm S}_n$ & $n$ & $2^4\cdot 3$ & 3\\
  
  $\mg_{11}$ & $11$ & $2^3\cdot 3$ & 3 \\

    $\mg_{12}$ & $12$ & $2^6\cdot 3$ & 3 or 5 \\

  $\mg_{23}$ & $23$ & $2^7\cdot 3^2$ & 3 or 5\\

  $\mg_{24}$ & $24$ & $2^9\cdot 3^2$ & 3, 5 or 9\\

  $\PSL_2(8)$ & $9$ & $2^2\cdot 3$ & --\\
  
  $\PGammaL_2(8)$ & $9$ & $2^2\cdot 3$ &  --\\

  $\PGammaL_2(32)$ & $33$ & $2^2$ & --\\  
    \hline
 \end{tabular}
 \caption{$4$-homogeneous groups for the proof of Theorem~\ref{thm:CTham}. }
 \label{tab:4homgroups}
 \end{center}
\end{table}

\subsection{Codes in Hamming graphs from permutation modules}\label{sec:poly}

In this section, we present examples of linear codes that are $s$-neighbour-transitive, for $s\geq 2$, arising from permutation modules. We construct these modules via polynomial algebras. Historically, polynomial algebras have been used to construct many interesting examples of codes, such as the generalised Reed--Muller codes and the projective Reed--Muller codes; see Definitions~\ref{def:genReedMuller} and~\ref{def:projReedMuller} below. We present a richer family of examples that may eventually lead to a classification of alphabet-affine, $2$-neighbour-transitive codes in $H(n,q)$ with minimum distance at least five: such a classification would be a huge strengthening of Theorem~\ref{thm:CTham}.

Throughout this section let $R=\F_q[x_1,\ldots,x_t]$, the ring of polynomials with coefficients in $\F_q$ in the variables $x_1,\ldots,x_t$. Each element of $R$ may be viewed as a function $\F_q^t\rightarrow \F_q$ and conversely, by Lagrange interpolation (see \cite[Theorem~1.7.1]{lidl1997finite}), every such function may be represented (in at least one way) by an element of $R$. A monomial $x_1^{a_1}\cdots x_t^{a_t}$ is said to have \emph{degree} $a_1+\cdots+a_t$ and the degree of a polynomial is the maximum value of the degrees of its constituent monomials. The generalised Reed--Muller codes arise as follows as subspaces of $R$; see \cite{delsarte1970generalized}.

\begin{definition}\label{def:genReedMuller}
 Let $k$ be an integer with $0\leq k\leq t(q-1)$. The \emph{$k$-th order $q$-ary generalised Reed--Muller code} $\RM_q(k,t)$ in $H(\F_q^t,\F_q)$ is the subspace of $R$ consisting of all polynomials of degree at most $k$. When $q=2$ these are simply called \emph{Reed--Muller} codes.
\end{definition}

The parameters of the generalised Reed--Muller codes are given in \cite[Theorem~5.4.1 and Corollary~5.5.4]{assmus1994designs}.
Let $d\in\{1,2,\ldots,q-1\}$ and let $R[d]$ be the subspace of $R$ consisting of all polynomials $f$ such that $f(ax_1,\ldots,ax_t)=a^d f(x_1,\ldots,x_t)$, for each $a\in\F_q$. In other words, $R[d]$ consists of all polynomials $f$ of $R$ such that each monomial of $f$ has non-zero degree equivalent to $d$ modulo $q-1$. This leads to a similar construction to that in the previous definition, but now involving $R[d]$. See \cite{sorensen1991projective} for more details. Note that we are taking $\N$ to be a set of representatives for the $1$-dimensional subspaces of $\F_q^t$ and then regarding the elements of $R[d]$ as functions $\N\rightarrow \F_q$.

\begin{definition}\label{def:projReedMuller}
 Let $k$ be an integer with $0\leq k\leq t(q-1)$, let $\N$ be a fixed set of representatives for the $1$-dimensional subspaces of $\F_q^t$ and let $d\in\{1,2,\ldots,q-1\}$ such that $d\equiv k\pmod{q-1}$. The \emph{$k$-th order $q$-ary projective Reed--Muller} code $\PRM_q(k,t)$ in $H(\N,\F_q)$ is the subspace of $R[d]$ consisting of all polynomials having degree at most $k$.
\end{definition}

The $\F_q\AGL_t(q)$-submodule structure of $R$ was determined by Sin \cite{sin2012affine}, while the $\F_q\GL_t(q)$-submodule structure of $R$ (and, in particular, of each $R[d]$) is determined in \cite{sin2000permutationmodules}. Any submodule of $R$ determines a code in $H(\F_q^t,\F_q)$ and any submodule of $R[d]$ determines a code in $H(\N,\F_q)$, where $\N$ is a set of representatives for the $1$-dimensional subspaces of $\F_q^t$. Indeed, the generalised Reed--Muller code $\RM_q(k,t)$ is an $\F_q\AGL_t(q)$-submodule of $R$ and the projective Reed--Muller code $\PRM_q(k,t)$ is an $\F_q\GL_t(q)$-submodule of $R[d]$, where $d\in\{1,2\ldots,q-1\}$ such that $d\equiv k\pmod{q-1}$. However, the submodule lattices of $R$ and $R[d]$ are, in general, considerably more complicated than the chains of submodules given by varying the parameter $k$ in $\RM_q(k,t)$ and $\PRM_q(k,t)$, respectively. In particular, if $r$ is such that $q=p^r$ for some $p$ prime then the submodule lattices of $R$ and $R[d]$ are parameterised by $r$-tuples of integers satisfying certain restrictions, see \cite[Theorems~A and~C]{sin2000permutationmodules}. 
%
The next two results show that, under certain conditions on the parameters $q, d$, many submodules of $R$ and $R[d]$ (including the generalised and the projective Reed--Muller codes and their duals) give examples of $2$-neighbour-transitive codes. Note that in Proposition~\ref{prop:moduleCodesAffineTopBinary}, since $q=2$ is prime, the codes arising are precisely the Reed--Muller codes and their duals. Recall that we denote by $T_\C$ the group of translations by elements of a linear code $\C$. In all the propositions in this section we assume that  the covering radius is at least $2$;  in  Remark~\ref{rem:smallCovRadHamming} we discuss the situation when this does not hold.

\begin{proposition}\cite{hawtin2023alphabetaffine} (or see \cite[Proposition~9.1.8]{DanHphdthesis})\label{prop:moduleCodesAffineTopBinary}
 Let $q=2$ and $\C$ be an $\F_2\AGL_t(2)$-submodule of $R$ such that $\C$ is a code with covering radius $\rho\geq 2$ in $H(\F_2^t,\F_2)$. Then $\C$ is $(G,2)$-neighbour-transitive, where $G=T_{\C}\rtimes \AGL_t(2)$.
\end{proposition}

\begin{proposition}\cite{hawtin2023alphabetaffine} (or see \cite[Proposition~9.2.3]{DanHphdthesis})\label{prop:moduleCodesProjective}
 Let $d\in\{1,2,\ldots,q-1\}$ with $\gcd(d,q-1)=1$, let $\N$ be a set of representatives for the $1$-dimensional subspaces of $\F_q^t$, and let $\C$ be an $\F_q\GL_t(q)$-submodule of $R[d]$ such that $\C$ is a code with covering radius $\rho\geq 2$ in $H(\N,\F_q)$. Then $\C$ is $(G,2)$-neighbour-transitive, where $G=T_{\C}\rtimes \GL_t(q)$.
\end{proposition}

\begin{remark}
 The conclusion of Proposition~\ref{prop:moduleCodesAffineTopBinary} is generally false if we instead consider $q\geq 3$, as is the conclusion of Proposition~\ref{prop:moduleCodesProjective} when $\gcd(d,q-1)\neq 1$; see \cite{hawtin2023alphabetaffine}, or \cite[Proposition~9.1.9]{DanHphdthesis}. Note also that the proofs of each of the propositions stated in this section proceed identically to the proofs of the respective references from \cite{DanHphdthesis}, though the statements here are more general than there. 
\end{remark}


Recalling from Definition~\ref{def:projReedMuller} that elements of $R[d]$ are polynomial functions, we can obtain further examples of $2$-neighbour-transitive codes by appropriate restrictions of the domain of these functions. For the first examples, we embed an affine space into a projective space and restrict to the $1$-dimensional subspaces corresponding to the affine points.

\begin{proposition}\cite{hawtin2023alphabetaffine} (or see \cite[Proposition~9.3.3]{DanHphdthesis})\label{prop:moduleCodesAffine}
 Let $d\in\{0,1,\ldots,q-1\}$ with $\gcd(d,q-1)=1$, and fix an embedding of $\AG_{t-1}(q)$ into $\PG_{t-1}(q)$, the points of which are the $1$-dimensional subspaces of $\F_q^t$. Let $\N$ be a set of representatives for the $1$-dimensional subspaces corresponding to the points of $\AG_{t-1}(q)$ and let $\C$ be an $\F_q\GL_t(q)$-submodule of $R[d]$ such that $\C$ is a code with covering radius $\rho\geq 2$ in $H(\N,\F_q)$. Then $\C$ is $(G,2)$-neighbour-transitive, where $G=T_{\C}\rtimes (\F_q^\times\times \AGL_{t-1}(q))$. 
\end{proposition}

It is worth briefly comparing the automorphism groups of the codes in Definition~\ref{def:genReedMuller} with the codes occurring in Proposition~\ref{prop:moduleCodesAffine}. There is a subgroup $\AGL_{t-1}(q)$ appearing in the automorphism group of the code $\RM_q(k,t-1)$ in $H(q^{t-1},q)$, and also inside the automorphism group of any code in $H(q^{t-1},q)$ arising as an $\F_q\GL_t(q)$-submodule of $R[d]$ as in Proposition~\ref{prop:moduleCodesAffine}. Moreover, $\AGL_{t-1}(q)$ acts faithfully on the entries of the Hamming graph in each case. However, the actions are not the same: in the former case $\AGL_{t-1}(q)$ occurs as a subgroup of the top group $L$, while this is not true in the latter case. This is the key difference that allows Proposition~\ref{prop:moduleCodesAffine} to be proved for more general values of $q$, while $q=2$ in Proposition~\ref{prop:moduleCodesAffineTopBinary}.


Next, we present two further infinite families of $2$-neighbour-transitive codes, the first arising from the Suzuki--Tits ovoids and the second from classical unitals. The Suzuki group $\Sz(q)$, where $q=2^{2f+1}$ for some positive integer $f$, acts $2$-transitively on the Suzuki--Tits ovoid consisting of $q^2+1$ points of the projective space $\PG_3(q)$, no three of which are collinear; see \cite[p.~250]{DM1996}. The unitary group $\PGU_3(q)$ acts $2$-transitively on the unital consisting of the $q^3+1$ isotropic points of $\PG_2(q^2)$ under a non-degenerate Hermitian form; see \cite[p.~248]{DM1996}.


\begin{proposition}\cite{hawtin2023alphabetaffine} (or see \cite[Proposition~9.4.6]{DanHphdthesis})\label{prop:moduleCodesSuzuki}
 Let $q=2^{2f+1}$, let $d\in\{0,1,\ldots,q-1\}$ with $\gcd(d,q-1)=1$, and let $\N$ be a set of representatives in  $\F_q^4$ for the $1$-dimensional subspaces corresponding to the points of the Suzuki--Tits ovoid in $\PG_3(q)$. Furthermore, let $\C$ be an $\F_q\GL_t(q)$-submodule of $R[d]$ such that $\C$ is a code with covering radius $\rho\geq 2$ in $H(\N,\F_q)$. Then $\C$ is $(G,2)$-neighbour-transitive, where $G=T_{\C}\rtimes (\F_q^\times\rtimes \Sz(q))$. 
\end{proposition}


\begin{proposition}\cite{hawtin2023alphabetaffine} (or see \cite[Proposition~9.4.8]{DanHphdthesis})\label{prop:moduleCodesUnitary}
 Let $q=2^f$, let $d\in\{0,1,\ldots,q-1\}$ with $\gcd(d,q-1)=1$, and let $\N$ be a set of representatives in $\F_{q^2}^3$ for the $1$-dimensional subspaces corresponding to the points of the classical unital in $\PG_2({q^2})$. Furthermore, let $\C$ be an $\F_q\GL_t(q)$-submodule of $R[d]$ such that $\C$ is a code with covering radius $\rho\geq 2$ in $H(\N,\F_q)$. Then $\C$ is $(G,2)$-neighbour-transitive, where $G=T_{\C}\rtimes (\F_q^\times\rtimes \PGU_3(q))$. 
\end{proposition}

\begin{remark}\label{rem:smallCovRadHamming}
 Note that for a code to be $2$-neighbour-transitive it must have covering radius at least $2$, hence this is an assumption in all the propositions of this section. However, since in each of the Propositions~\ref{prop:moduleCodesAffineTopBinary}--\ref{prop:moduleCodesUnitary} the group $G$ acts transitively on both $\Gamma_1({\b 0})$ and $\Gamma_2({\b 0})$, it follows from Proposition~\ref{prop:coveringRadiusLessThan2} that any non-trivial code $\C$ arising from the respective submodule in the relevant Hamming graph, but having covering radius $\rho\leq 1$, is either perfect with $(\rho,\delta)=(1,3)$, or has $\rho=1$ and $\delta=2$. Moreover, if $\Gamma=H(n,q)$ then either $\Gamma_1({\b 0})$ and $\Gamma_2({\b 0})$ both contain no pairs of adjacent vertices (when $q=2$) or each contains a pair of adjacent vertices (when $q\geq 3$). This implies that Proposition~\ref{prop:coveringRadiusLessThan2}(3)(b) does not occur. In particular, any relevant (linear) code with covering radius $\rho\leq 1$ is either a perfect Hamming code\footnote{To see this: by \cite[Theorem~37, Chapter~6] {macwilliams1978theory} a perfect linear code $\C$ with covering radius $1$ in $H(n,\F_q)$ necessarily has length $n=(q^k-1)/(q-1)$, dimension $k$ and minimum distance $3$. The condition `minimum distance $3$' implies that each column of a parity-check matrix $H$ for $\C$ is non-zero, and no pair of columns of $H$ is linearly dependent. This implies that the columns of $H$ are a set of representatives for the $1$-dimensional subspaces of $\F_q^k$, {\em i.e.}, that $\C$ is a Hamming code.} or is the binary repetition code, and is thus known.
\end{remark}

Assmus and Key \cite[Section~5.7]{assmus1994designs} construct and analyse the \emph{subfield subcodes} of the generalised and projective Reed--Muller codes. Further examples of $2$-neighbour-transitive codes may be obtained in a similar manner from submodules of $R$ and $R[d]$; see \cite{hawtin2023alphabetaffine} or \cite[Section~9.5]{DanHphdthesis}.

The assumption in Propositions~\ref{prop:moduleCodesAffineTopBinary}--\ref{prop:moduleCodesUnitary} that $\C$ is an $\F_q\GL_t(q)$-submodule of $R[d]$ is more restrictive than necessary. However we have stated the results in this way because, although the lattice of $\F_q\GL_t(q)$-submodules of $R[d]$ is known (see \cite{sin2000permutationmodules}), we wonder whether there may be additional subspaces invariant under the subgroups of $\GL_t(q)$ occurring in these propositions, which may yield new interesting codes.

\begin{problem}\label{prob:2NTmodules}
 Determine more information about the $G_{\b 0}$-submodule structure of $R[d]$, where $G_{\b 0} = \F_q^\times\times \AGL_{t-1}(q)$, $ \F_q^\times\rtimes \Sz(q)$ or $\F_q^\times\rtimes \PGU_3(q)$ in Proposition~\ref{prop:moduleCodesSuzuki},  \ref{prop:moduleCodesAffine} or~\ref{prop:moduleCodesUnitary}, respectively.
\end{problem}

As alluded to earlier, a reasonable amount is known about the minimum distance of each of the generalised and projective Reed-Muller codes, see \cite[Section~5.5]{assmus1994designs}. It would be nice to have similar results for the other codes  discussed in this section arising from submodules.

\begin{problem}\label{prob:submoduleParams}
 Find the minimum distances of the codes from submodules of $R$ and $R[d]$  in Propositions~\ref{prop:moduleCodesAffine}--\ref{prop:moduleCodesUnitary}.
\end{problem}

It should be remarked also that some of the codes in this section, and their subfield subcodes, are related to codes arising from incidences between subspaces of differing dimensions in projective and affine geometries; see \cite[Section~5.6]{assmus1994designs} and \cite[Section~8]{sin2000permutationmodules}.
Interesting work has been done concerning the geometric structure of codewords of low weight in some of these cases; see, for example, \cite{adriansen2021smallweight,deboeck2012smallweight,lavrauw2008incidence}.


\section{Codes in Kneser graphs}\label{sec:KneserSection}


Neighbour-transitive and $2$-neighbour-transitive codes have recently been studied in the Kneser graphs \cite{crnkovic2022kneser}. The next result is a classification of $2$-neighbour-transitive codes $\C$ in Kneser graphs with minimum distance $\delta(\C)\geq 5$. Recall from Definition~\ref{defKneser} that in a Kneser graph $K(\V,k)$, the cardinality $v=|\V|$ is at least $2k+1$. Note also that, for $|\V|=23$, by an \emph{endecad} we mean a subset of $\V$ such that its characteristic vector corresponds to a weight $11$ codeword of the perfect binary Golay code in the Hamming graph $H(\V,2)$ (see \cite[Page 71]{conway1985atlas}).

\begin{theorem}\cite[Theorem~1.2]{crnkovic2022kneser}\label{thm:Kneser2NT}
 Let $\C$ be a $2$-neighbour-transitive code in $\Gamma=K(v,k)$ with minimum distance $\delta\geq 5$. Then $v=2k+1$, and hence $\Gamma$ is the odd graph $O_{k+1}$, and one of the following holds.
 \begin{enumerate}[$(1)$]
  \item $\Aut(\C)\cong\mg_{23}$ with $v=23$ and $\C$ consists of the endecads.
  \item $\Aut(\C)\cong\PGaL_d(2)$, where $d\geq 5$, $\Omega$ is the set of all points, and $\C$ is the set of all hyperplanes, in $\PG_{d-1}(2)$.
 \end{enumerate}
\end{theorem}

Recall from Example~\ref{ex:jk} that any code in a Kneser graph, and hence also any code in an odd graph, may be regarded as a code in a Johnson graph. A similar situation exists  for codes in Johnson graphs which can also be viewed as codes in binary Hamming graphs; and in this case there is a clear connection between the minimum distances of the codes, see \cite[Section 1.3]{neunhoffer2014sporadic}. The following lemma gives a similar relationship between the minimum distances for codes in Kneser graphs and the same codes in the corresponding Johnson graphs. This lemma is key to proving Theorem~\ref{thm:Kneser2NT}, where the proof proceeds by reducing to the case of odd graphs (see \cite[Theorem~3.1]{crnkovic2022kneser}) and then applying results about neighbour-transitive codes in Johnson graphs \cite{praeger2021codes} along with the following Lemma~\ref{oddtojohnson}.

\begin{lemma}\cite[Lemma~3.2]{crnkovic2022kneser}\label{oddtojohnson}
 Let $\C$ be a $2$-neighbour-transitive code in $O_{k+1}$ with minimum distance $\delta\geq 5$. Then $\C$ is also a code in the Johnson graph $J(2k+1,k)$, with the same vertex set as $O_{k+1}$, and $\C$ is neighbour-transitive in $J(2k+1,k)$ with minimum distance $\delta'\geq 3$. 
\end{lemma}


 

Given the classification in Theorem~\ref{thm:Kneser2NT}, the paper \cite{crnkovic2022kneser} then considers neighbour-transitive codes $\C$ in Kneser graphs $K(\V,k)$ in general, roughly separating the analysis into cases where the action of $\Aut(\C)$ on $\V$ is: intransitive, transitive but imprimitive, and primitive. The following example and theorem deal with  the intransitive case, and provide a useful application of the concept of types (see Definition~\ref{GinvDef}) and Lemma~\ref{invlemma}.

\begin{example}\label{ex:KneserIntrans}
 Let $\Gamma=K(\V,k)$, let $a,b,c$ and $d$ be non-negative integers such that $a\geq c$, $b\geq d$, $a+b=v=|\V|$ and $c+d=k$. Let $\V$ be the disjoint union $A\cup B$ with $|A|=a$ and $|B|=b$. For $\alpha\in V(\Gamma)$ let $\iota(\alpha)=(|\alpha\cap A|,|\alpha\cap B|)$. Define 
 \[
  \C_{{\rm int}}(a,b;c,d)=\{\alpha\in V(\Gamma)\mid \iota(\alpha)=(c,d)\}.
 \]
 Then $\C_{{\rm int}}(a,b;c,d)$ is neighbour-transitive if $a,b,c,d$ are as in one of the lines of Table~\ref{tab:KneserIntrans} (see \cite[Lemma~4.2]{crnkovic2022kneser}).
\end{example}

\begin{theorem}\cite[Theorem~1.3]{crnkovic2022kneser}\label{thm:KneserIntrans}
 Let $\C$ be a non-trivial neighbour-transitive code in $K(\V,k)$ with minimum distance $\delta$ and suppose that $\Aut(\C)$ acts intransitively on $\V$. Then $\Aut(\C)$ has precisely two non-empty orbits on $\V$, say $A$ and $B$, and $\C$ is equivalent to a subcode of one of the codes in Example~\ref{ex:KneserIntrans}.
\end{theorem}

\begin{table}
 \begin{center}
 \begin{tabular}{c|ccccc}
  \hline
  Line & $a$ & $b$ & $c$ & $d$ & $\delta$\\
  \hline
   1 & $1$ & $v-1$ & $0$ & $k$ & $1$ \\
   2 & $2e$ & $2f+1$ & $e$ & $f$ & $1$ \\
   3 & $< k$ & $v-a$ & $a$ & $k-a$ & $2$ \\
  \hline
 \end{tabular}
 \caption{Conditions on the parameters of $\C_{{\rm int}}(a,b;c,d)$ in Example~\ref{ex:KneserIntrans} ensuring it is neighbour-transitive with minimum distance $\delta$.}
 \label{tab:KneserIntrans}
 \end{center}
\end{table}


The following example shows that there may indeed be proper subcodes of the codes given in Example~\ref{ex:KneserIntrans} which are neighbour-transitive and have automorphism groups intransitive on $\V$.

\begin{example}\cite[Example~4.6]{crnkovic2022kneser}\label{ag32tetrahedronEx}
 Let $\V$ be the disjoint union $A\cup B$, where $A$ is the set of points of the affine geometry $\AG_3(2)$ and $|B|=5$. Furthermore, let $\T$ be the set of all \emph{tetrahedrons} of $A$, where a tetrahedron is set of $4$ points of $\AG_3(2)$ that do not form an affine plane, and let $\C$ be the code in $O_7=K(13,6)$ consisting of all vertices $\alpha$ such that $\alpha\cap A\in\T$ and $|\alpha\cap B|= 2$. Note that $\C$ is a proper subcode of $\C_{{\rm int}}(8,5;4,2)$ and is neighbour-transitive (see \cite[Lemma~4.7]{crnkovic2022kneser}).
\end{example}

The next theorem concerns the case where $\C$ is neighbour-transitive and $\Aut(\C)$ acts transitively on $\V$. Note that a $2$-homogeneous permutation group is primitive, \cite[Lemma 2.30]{PS2018}.

\begin{theorem}\cite[Theorem~1.7]{crnkovic2022kneser}\label{thm:knesernotodd2hom}
 Let $\C$ be a neighbour-transitive code in $K(v,k)$ with minimum distance $\delta\geq 3$.
 \begin{enumerate}[$(1)$]
     \item If $\Aut(\C)$ is transitive and imprimitive on $\V$ then $v=2k+1$ so $K(v,k)$ is the odd graph $O_{k+1}$; and 
     \item  if $\Aut(\C)$ is primitive on $\V$, then then either $v=2k+1$ and $K(v,k)=O_{k+1}$, or $\Aut(\C)$ is $2$-homogeneous on $\V$.
 \end{enumerate}
\end{theorem}


Noting that the $2$-homogeneous groups are classified   (see \cite[Section 7.7 and Theorem 9.4B]{DM1996}), we pose the following problems.

\begin{problem}\cite[Problem~1.5]{crnkovic2022kneser}\label{prob:Kneser2hom}
 Classify the neighbour-transitive codes $\C$  in $K(\V,k)$ such that $\delta(\C)\geq 3$ and $\Aut(\C)$ acts $2$-homogeneously on $\V$.
\end{problem}


\begin{problem}\label{prob:OddPrim}
 Find examples of neighbour-transitive codes $\C$  in $K(2k+1,k)=O_{k+1}$  such that $\delta(\C)\geq3$, and $\Aut(\C)$ is primitive  on $\V$ but not $2$-homogeneous (see also \cite[Problem~1.7]{crnkovic2022kneser}).
\end{problem}

Given the above results and open problems, we should comment briefly on neighbour-transitive codes $\C$ in odd graphs where $\Aut(\C)$ is imprimitive on $\V$. The next example and theorem are again applications of types and Lemma~\ref{GinvDef}.
Given a multiset $M$ we write $M=\{b_1^{a_1},\ldots,b_s^{a_s}\}$ where each $b_i$ is an element of $M$ that occurs with multiplicity $a_i$, for $i=1,\ldots,s$. For example, the multiset $\{0,1,1,2,2,2\}$ could be written as $\{0^1,1^2,2^3\}$. 

\begin{example}\label{ex:OddImprim}
 Let $\Gamma=K(2k+1,k)=O_{k+1}$ and let $\B=\{B_1,\ldots,B_a\}$ be a partition of $\V$ into $a$ blocks each having size $b$. For a vertex $\alpha\in V(\Gamma)$, let $\iota(\alpha)$ be the multiset $ \{\alpha\cap B_1,\ldots,\alpha\cap B_a\}$. For a multiset $M$, define
 \[
  \C_{{\rm imp}}(a,b;M)=\{\alpha\in V(\Gamma)\mid \iota(\alpha)=M\}.
 \]
 Then, by \cite[Lemma~5.2]{crnkovic2022kneser}, $\C_{{\rm imp}}(a,b;M)$ is neighbour-transitive if and only if $M$ is as in one of the lines of Table~\ref{tab:OddImprim}. 
\end{example}

\begin{table}
 \begin{center}
 \begin{tabular}{c|cc}
  \hline
  Line & $M$ & $\delta$ \\
  \hline
   1 & $\{((b-1)/2)^{(a+1)/2},((b+1)/2)^{(a-1)/2}\}$ & $1$ \\
   2 & $\{0^{(a-1)/2},(b-1)/2,b^{(a-1)/2}\}$ & $1$  \\
   3 & $\{b^{a_0},b_1^{a_1}\}$ & $2$  \\
  \hline
 \end{tabular}
 \caption{Multisets $M$ for which $\C_{{\rm imp}}(a,b;M)$, as in Example~\ref{ex:KneserIntrans}, is neighbour-transitive with minimum distance $\delta$.}
 \label{tab:OddImprim}
 \end{center}
\end{table}

\begin{theorem}\cite[Theorem~1.6]{crnkovic2022kneser}\label{thm:OddImprim}
 Let $\C$ be a non-trivial neighbour-transitive code in $K(2k+1,k)=O_{k+1}$ such that $\Aut(\C)$ acts transitively but imprimitively on $\V$. Then $\C$ is equivalent to a subcode of one of the codes in Example~\ref{ex:OddImprim}.
\end{theorem}

Note that there are currently no known examples of codes $\C$ where $\Aut(\C)$ acts imprimitively on $\V$ and $\C$ is a proper subcode of a code from Example~\ref{ex:OddImprim}. Hence we finish this section with the following research problem. 

\begin{problem}\cite[Problem~1.6]{crnkovic2022kneser}\label{prob:oddImprimitive}
 Find new examples of codes satisfying Theorem~\ref{thm:OddImprim}, or prove that all examples are equivalent to a code in Example~\ref{ex:OddImprim}.
\end{problem}

\section{Codes in incidence graphs of generalised quadrangles}\label{sec:GQs}

A \emph{generalised quadrangle} is an incidence structure\footnote{Recall the definition of an incidence structure from Section~\ref{sec:incidenceGraphsPrelim}.} $\Qu=(\P,\L,\I)$ such that:
\begin{enumerate}[(1)]
 \item Each point is incident with $t+1$ lines ($t\geq 1$) and two distinct points are incident with at most one line.
 \item Each line is incident with $s+1$ points ($s\geq 1$) and two distinct lines are incident with at most one point.
 \item If $p$ is a point and $L$ is a line not incident with $p$, then there is a unique pair $(q,M)\in\P\times \L$ for which $p\,\I\,M\,\I\,q\,\I\,L$.
\end{enumerate}
A generalised quadrangle $\Qu$ satisfying the above axioms is said to have \emph{order} $(s,t)$ and has $(s+1)(st+1)$ points and $(t+1)(st+1)$ lines; and $\Qu$ is called \emph{thick} if both $s,t\geq 2$. The dual of a generalised qudrangle of order $(s,t)$ is a generalised quadrangle of order $(t,s)$. For further background on generalised quadrangles see \cite{paynefinite}.


Let $\Qu$ be a generalised quadrangle and let $\Gamma$ be its incidence graph (see Definition~\ref{defIncidenceStrucGraphs}). Then $\Gamma$ is bipartite, has degrees $s+1$ and $t+1$, diameter $4$ and girth $8$. 
Note that, since $\Gamma$ has diameter $4$, a code $\C$ in $\Gamma$ has minimum distance $\delta(\C)\leq 4$.  

An \emph{ovoid} (respectively, a \emph{partial ovoid}) of a generalised quadrangle $\Qu=(\P,\L,\I)$ is a subset $\O$ of $\P$ such that each line $\ell\in\L$ is incident with exactly one (respectively, at most one) point of $\O$.
Dually, a \emph{spread} (respectively, a \emph{partial spread}) of a generalised quadrangle $\Qu=(\P,\L,\I)$ is a subset $\S$ of $\L$ such that each point $p\in\P$ is incident with exactly one (respectively, at most one) line of $\S$.
A partial ovoid (respectively, spread) is called \emph{maximal} if there is no partial ovoid (spread) properly containing it. In particular, an ovoid (spread) is a maximal partial ovoid (spread). These geometric conditions can be reformulated in the language of coding theory as follows.

\begin{lemma}\cite[Lemma~3.6]{crnkovic2022neighbour}\label{ovoidorspread}
 Let $\C$ be a code in a generalised quadrangle $\Qu$ with minimum distance $\delta= 4$ and covering radius $\rho$. Then the following hold:
 \begin{enumerate}[$(1)$]
  \item $\C$ is a partial ovoid or a partial spread of $\Qu$.
  \item $\C$ is a maximal partial ovoid or a maximal partial spread of $\Qu$ if and only if $\rho\leq 3$.
  \item $\C$ is an ovoid or spread of $\Qu$ if and only if $\rho=2$.  
 \end{enumerate}
\end{lemma}

Now we consider 
\emph{classical} generalised quadrangles, which are associated with certain classical groups; see \cite[Chapter 3]{paynefinite} for their constructions and \cite[Sections~3.5.6 and 3.6.4]{wilson2009finite} for more about their automorphism groups. 
Those that arise in our next result are the \emph{symplectic generalised quadrangle} $\W_3(q)$ and the  \emph{hermitian generalised quadrangle} $H_3(q^2)$,  defined as follows: let $V$ be the underlying vector space of the projective geometry $\PG_3(q^\tau)$ equipped with a non-singular symplectic or hermitian form $f$, where $\tau=1$ or $2$, respectively. 
The points are the totally isotropic $1$-dimensional subspaces, and the lines are are the totally isotropic $2$-dimensional subspaces of $V$, with incidence given by symmetrised inclusion. If $q$ is a square then a regular spread of $\W_3(q)$ can obtained by embedding $\W_1(q^2)$ into it, see \cite[Section~3.2]{bamberg2009classification}; and a classical ovoid of $\He_3(q^2)$ can be constructed by taking the absolute points of a non-degenerate unitary polarity, see \cite[Section~3.1]{bamberg2009classification}. The associated codes were shown to be neighbour-transitive in \cite[Lemmas~4.2 and 4.4]{crnkovic2022neighbour}.

\begin{theorem}\cite[Theorem~4.5]{crnkovic2022neighbour}\label{thm:NTovoidsspreads}
 Let $\C$ be a neighbour-transitive code with minimum distance $4$ and covering radius $\rho=2$ in the incidence graph of a thick classical generalised quadrangle $\Qu$ and assume that $\Aut(\C)$ is insoluble. Then $\C$ is equivalent to one of the following:
 \begin{enumerate}[$(1)$]
  \item A regular spread of $\W_3(q)$, where $q$ is a square.
  \item A classical ovoid of $\He_3(q^2)$.
 \end{enumerate}
\end{theorem}




Some sporadic examples of maximal partial spreads in $\W_3(q)$ are given in Example~\ref{maxspreadconstr}. 
We note that, if $q$ is even then $\W_3(q)$ is self dual, while if $q$ is odd then the dual of $\W_3(q)$ is the classical generalised quadrangle $\Quu_4(q)$. Thus for odd $q$, a maximal partial spread in $W_3(q)$ is a maximal partial ovoid in $\Quu_4(q)$. 
It is a conjecture of Thas \cite[Conjecture, p.~13]{thas2004symmetry} that when $q$ is sufficiently large then there are no maximal partial ovoids of size $q^2-1$ in $\Quu_4(q)$ (the dual of $\W_3(q)$).


\begin{example}\cite[Example~5.4]{crnkovic2022neighbour}\label{maxspreadconstr}
 Let $V\cong\F_q^4$ with symplectic form $f$ such that $f(x,y)=x_1y_2-x_2y_1-x_3y_4+x_4y_3$. Let $q, G$ be as in one of the rows of Table~\ref{table:sharplytransgps}, so $G\leq \SL_2(q)$ (represented as $2\times 2$ matrices) and $G$  is sharply transitive on the non-zero vectors of $\F_q^2$. Let $\C$ be the following set of $2$-dimensional  subspaces of $V$, where each is represented as the row-space of a $2\times 4$ matrix:
 \[
  \C=\{[I \ A]\mid A\in G\},
 \]
 where $I$ is the $2\times 2$ identity matrix. Letting $x$ be the first row of $[I \ A]$ and $y$ be the second row, we have $f(x,y)={\rm det} I - {\rm det} A$. Thus the row-space of $[I \ A]$ is an isotropic $2$-space if and only if ${\rm det} I - {\rm det} A=0$, that is, ${\rm det} A=1$. Since $A\in \SL_2(q)$, the code $\C$ is indeed a subset of lines of $\W_3(q)$. By \cite[Lemma~5.5]{crnkovic2022neighbour}, $\C$ is a neighbour-transitive maximal partial ovoid of $\W_3(q)$. Note that this is the dual of a construction for maximal partial spreads given in \cite{coolsaet2013known}.
\end{example}

\begin{table}
 \begin{center}
 \begin{tabular}{c|ccccc}
  \hline
  $q$ & $2$ & $3$ & $5$ & $7$ & $11$ \\ 
  \hline
  $G$ & $\GL_1(4)$ & $Q_8$ & $2.\alt_4$ & $2.\s_4$ & $\SL_2(5)$\\
  \hline
 \end{tabular}
 \caption{Subgroups of $\SL_2(q)$ of order $q^2-1$ (see \cite[Chapter 3, Section 6]{suzuki1982group}). }
 \label{table:sharplytransgps}
 \end{center}
\end{table}



We collect together, in Theorem~\ref{thm:GQsmainresult}, information about neighbour-transitive partial ovoids and spreads in $\W_3(q)$, focusing mainly on maximal partial ovoids and spreads. Examples for each case can be found in \cite{crnkovic2022neighbour}.


\begin{theorem}\cite[Theorem~1.2]{crnkovic2022neighbour}\label{thm:GQsmainresult}
 Let $\C$ be a neighbour-transitive code with minimum distance $4$ and covering radius $\rho$ in the incidence graph of the generalised quadrangle $\W_3(q)$. 
 Then, for each line of Table~\ref{tab:GQsmainresult}, if the `Conditions' hold then the `Conclusions' also hold. Moreover, in line (1), the converse assertion is also valid. 
\end{theorem}

\begin{table}
 \begin{center}
 \begin{tabular}{l|ll}
  \hline
    & Conditions & Conclusions  \\ \hline
   (1) &  $\rho=2$ & $\C$ is equivalent to a regular spread \\
   (2) & $|\C|=q^2$ & $\rho=4$, and $\C$ can be extended to a spread or an ovoid \\
   (3) & $|\C|=q^2-1$ and $\rho=3$ & $q\in\{2,3,5,7,11\}$ and $\C$ is equivalent to a code in Example~\ref{maxspreadconstr} \\
   (4) & $|\C|=q+1$ and $\rho=3$ & $\C$ is equivalent to the set of points on a hyperbolic line \\
   (5) & $q=\rho=3$  & $\C$ as in line (3) or  (4), or $\C$ is equivalent to the sporadic code\\
      &               &given in \cite[Example~5.3]{crnkovic2022neighbour} with $|\C|=5$ \\
  \hline
 \end{tabular}
 \caption{Results table for Theorem~\ref{thm:GQsmainresult}}
 \label{tab:GQsmainresult}
 \end{center}
\end{table}


\begin{remark}
 The statement of Theorem~\ref{thm:GQsmainresult} differs slightly from that of \cite[Theorem~1.2]{crnkovic2022neighbour}: in lines  (3) and  (5) of Table~\ref{tab:GQsmainresult} we have $\rho=3$ as a hypothesis rather than a conclusion (correcting a mistake in \cite[Theorem~1.2]{crnkovic2022neighbour}), and the conclusion in line (5) of Table~\ref{tab:GQsmainresult} is stronger than in  \cite[Theorem~1.2]{crnkovic2022neighbour}, reflecting the discussion preceding \cite[Conjecture~1.4]{crnkovic2022neighbour}.
 
\end{remark}

Generalised quadrangles are examples of a broader class of incidence structures called \emph{polar spaces}. We pose the following open problem.

\begin{problem}\label{prob:GQsandPolarSpaces}
 Investigate $s$-neighbour-transitive codes in the incidence graphs of other classical generalised quadrangles and, more generally, in other classical polar spaces.
\end{problem}

\section{Coda: final reflections and summary of open problems}


Our aim in the chapter has been to outline the state-of-the-art regarding our understanding of $s$-neighbour-transitive codes in various graphs. This is an area of active research and we have posed several open problems throughout the chapter; these are summarised for reference in Table~\ref{tab:problems}. In the remainder of this section we briefly reflect on the most significant achievements and major open problems we have covered. We also mention a few interesting related results and areas of research which for reasons of space we could not discuss in detail. 

\begin{table}
 \begin{center}
 \begin{tabular}{ll}
  \hline
  Problem & Topic \\
  \hline
    \ref{prob:covRadOne} & Codes with covering radius $1$.\\
    \ref{prob:sElusive} & $s$-Elusive codes. \\
    \ref{prob:sdt} & $s$-Distance-transitive quotient graphs. \\
    \ref{prob:bilinearForms} & Codes in bilinear forms graphs. \\
    \ref{prob:freqPermArrays} & Frequency permutation arrays. \\
    \ref{prob:binaryDeltaEquals4} & Codes in $H(n,2)$ with minimum distance $4$. \\
    \ref{prob:binaryCT} & Binary completely transitive codes. \\
    \ref{prob:2NTmodules} & $2$-Neighbour-transitive codes from submodules. \\
    \ref{prob:submoduleParams} & Parameters of codes from submodules. \\
    \ref{prob:Kneser2hom} & $2$-Homogeneous actions and Kneser graphs. \\
    \ref{prob:OddPrim} & Primitive actions and odd graphs. \\
    \ref{prob:oddImprimitive} & Imprimitive actions and odd graphs. \\
    \ref{prob:GQsandPolarSpaces} & $s$-Neighbour-transitive codes in polar spaces. \\
  \hline
 \end{tabular}
 \caption{References for open problems stated in this chapter and a rough description of each.}
 \label{tab:problems}
 \end{center}
\end{table}

Regarding codes in Hamming graphs, Theorem~\ref{thm:CTham} effectively gives an upper bound on $\min\{e,s\}$, where $e$ is the error-correction capacity of an $s$-neighbour-transitive code. Furthermore, progress has been made classifying completely transitive codes with minimum distance at least $5$ in the binary case (see Section~\ref{sec:bin}), though there is still significant work to be done here (see Problem~\ref{prob:binaryCT}). The results of Sections~\ref{sec:aaff} and~\ref{sec:poly} pave the way towards a deeper understanding of $2$-neighbour-transitive codes with minimum distance at least $5$ in $H(n,q)$ when $q\geq 3$; these may lead in the future towards classification results for completely transitive codes with alphabet size larger than $2$.

Turning to other graphs, classification results have been obtained for neighbour-transitive codes in Johnson graphs (see \cite{praeger2021codes}) and progress has been made on neighbour-transitive codes in Kneser graphs (see Section~\ref{sec:KneserSection}). There are still open problems related to codes in each of these families of graphs, as there are for numerous other families of distance-regular and distance-transitive graphs (see, for instance, Sections~\ref{sec:qAnaloguesPrelim} and~\ref{sec:GQs}).

Many of the results we have stated for codes in Hamming graphs assume some small lower bound (typically $5$) on the minimum distance of a code. That is not to say that codes with smaller minimum distances are not interesting. For example, recently Borges, Rif\`{a} and Zinoviev \cite{borges2023newCT} investigated the complete transitivity of certain ``supplementary'' codes in $H(n,q)$. These are constructed via a concatenation method previously introduced by the same authors. They find several infinite families of codes with minimum distance $3$ and covering radius $1$ or $2$. They also conjecture that these are all the completely transitive codes that may be obtained via their construction -- their conjecture is informed by computational results on the sizes of the automorphism groups of some of the codes.

A \emph{maximum distance separable} (MDS) code is a code $\C$  in the Hamming graph $H(n,q)$ that meets the Singleton bound \cite[Theorem~11, Chapter~1]{macwilliams1978theory}, that is, if $|\C|=q^k$ and $\delta$ is the minimum distance of $\C$ then $n=\delta+k-1$ (see \cite{thas1992mds}). By \cite[Theorem~3]{thas1992mds}, a linear MDS code $\C$ is ``equivalent'' to an \emph{$n$-arc} in the projective space $\PG_{k-1}(q)$, that is, the column vectors of a generator matrix for $\C$ are representatives for a set of $n$ points in $\PG_{k-1}(q)$, with $n\geq k$, such that no subset of $k$ points is contained in a hyperplane. The archetypal example of an MDS code is a Reed--Solomon code $\C$ in $H(q+1,q)$, which corresponds to a geometric object known as a \emph{normal rational curve} (see \cite[Section~2]{thas1992mds}). In fact, such a Reed--Solomon code is equivalent to the projective Reed--Muller code $\C=\PRM_q(k-1,2)$ and, by Proposition~\ref{prop:moduleCodesProjective}, is $2$-neighbour-transitive when $\gcd(k-1,q-1)=1$ with automorphism group $T_\C\rtimes \GammaL_2(q)$ (see also \cite{dur1987reedsolomon}, noting that a different notion of automorphism group is used there). The \emph{MDS conjecture} states that if $\C$ is an MDS code of size $q^k$ in $H(n,q)$ and $4\leq k\leq q-3$ then $n\leq q+1$. Ball \cite{ball2011mds} proved the MDS conjecture when $q$ is a prime; but it is still open for non-prime  $q$. More recently, additive MDS codes have been classified over small fields \cite{ball2020additiveMDS} and certain additive MDS codes have been shown to be equivalent to linear codes \cite{adriansen2023additiveMDS}.

A code $\C$ in a graph is called \emph{propelinear}\footnote{The reader should note that we have stated this definition in the language of this chapter. Much of the literature regarding propelinear codes uses a different, but equivalent, definition.} if $\Aut(\C)$ contains a subgroup $H$ such that $H$ acts regularly on $\C$ (that is $H$ is transitive on $\C$ and codeword stabilisers fix $\C$ pointwise). In particular, the automorphism group of any linear code $\C$ in $H(n,q)$ contains the group of translations by codewords; this group acts regularly on $\C$, and thus each linear code is propelinear. Propelinear codes have primarily been studied in the Hamming graphs: for example, Rif\`{a} and Pujol \cite{rifapujol1997prop} studied a subclass of propelinear codes, known as \emph{translation-invariant} propelinear codes. In addition, many interesting codes in $H(2k,q)$ have been constructed as additive codes in $\Z_4^k$, and are thus propelinear (see \cite{hammons1994z4}). This is an active research area, with new interesting examples of propelinear codes still being discovered (see, for instance, \cite{armario2023butson}). It would be interesting to study propelinear codes in other distance-regular graphs.